\documentclass[12pt]{article}
\usepackage{amsmath,mathtools,amssymb,amsfonts,amsthm,mathrsfs,
url,xpatch,hyphenat}
\usepackage[shortlabels]{enumitem}

\usepackage[dvipsnames]{xcolor}        
\usepackage{tikz-cd}

\newcommand{\TE}{\mathsf{TE}}

\newcommand{\calGrEs}{\mathscr{G}^r_{\E^*}}

\newcommand{\A}{\mathsf{A}}
\newcommand{\B}{\mathsf{B}}
\newcommand{\TM}{\mathsf{TM}}
\newcommand{\TQ}{\mathsf{TQ}}
\newcommand{\TsM}{\T^*\M}
\newcommand{\M}{\mathsf{M}}
\newcommand{\Q}{\mathsf{Q}}
\newcommand{\F}{\mathsf{F}}
\newcommand{\E}{\mathsf{E}}
\newcommand{\R}{\mathsf{R}}

\newcommand{\T}{\mathsf{T}}
\newcommand\reg{$^{\textsf{\tiny\textregistered}}$}

\newcommand{\rd}{\xrightarrowt{\delta}}
\newcommand{\D}{\mathsf{D}}
\newcommand{\Xnh}{X^{\mathrm{nh}}_\D}

\newcommand\Xlin{X^{\mathrm{lin}}}
\newcommand\Xaff{X^{\mathrm{aff}}}
\newcommand{\nx}{\nabla_{X_0}}

\newcommand{\C}{\mathrm{C}}
\newcommand{\HE}{\mathsf{HE}}
\newcommand{\VE}{\mathsf{VE}}
\newcommand{\G}{\mathbb{G}}

\newcommand{\V}{\mathsf{V}}

\newcommand{\scrL}{\mathscr{L}}
\newcommand{\scrLX}{\scrL_{X_0}}
\newcommand{\scrLXk}{\scrL^k_{X_0}}
\newcommand{\scrLXh}{\scrL_{{X_0^\mathrm{h}}}}
\newcommand\scrLAe{\scrL_{{A^\mathrm{e}}}}

\newcommand{\scrLbv}{\scrL_{b^\mathrm{v}}}

\newcommand{\scrLXaff}{\scrL_{X^{\mathrm{aff}}}}

\newcommand{\bbR}{\mathbb{R}}

\newcommand{\calG}{\mathcal{G}}
\newcommand\bbT{\mathbb{T}}
\newcommand\bbI{\mathbb{I}}
\newcommand\bbG{\mathbb{G}}
\newcommand\sS{\mathsf{S}}
\newcommand\what[1]{\widehat{#1}\null}
\DeclareMathOperator\Span{span}
\DeclareMathOperator\End{End}
\DeclareMathOperator\Fl{Fl}
\DeclareMathOperator\hlft{hl}
\DeclareMathOperator\vlft{vl}

\newcommand\dd{\mathrm{d}}
\newcommand\ddt{\frac{\dd}{\dd t}}
\newcommand\calC{\mathcal{C}}
\DeclareMathOperator\image{image}

\newcommand\nabd{\stackrel{\mathclap{\scriptscriptstyle{\D}}}{\nabla}}
\newcommand\nabdp{\stackrel{\mathclap{\scriptscriptstyle{\D^\perp}}}{\nabla}}
\newcommand\nabg{\stackrel{\mathclap{\scriptscriptstyle{\bbG}}}{\nabla}}
\newcommand\nablaD{\stackrel{\mathclap{\scriptscriptstyle{\D^\perp}}}{\nabla^*}}
\newcommand\nablaDns{\stackrel{\mathclap{\scriptscriptstyle{\D^\perp}}}{\nabla}}

\newlength{\arrow}
\newcommand*{\xrightarrowt}[1]{\xrightarrow{\mathmakebox[\arrow]{#1}}}

\makeatletter
\newcommand{\xdashrightarrow}[2][]{\ext@arrow 0359\rightarrowfill@@{#1}{#2}}
\def\rightarrowfill@@{\arrowfill@@\relax\relbar\rightarrow}
\def\arrowfill@@#1#2#3#4{%
	$\m@th\thickmuskip0mu\medmuskip\thickmuskip\thinmuskip\thickmuskip
	\relax#4#1
	\xleaders\hbox{$#4#2$}\hfill
	#3$%
}
\makeatother

\theoremstyle{plain} 
{\swapnumbers
\newtheorem{theorem}{Theorem}[section]
\newtheorem{lemma}[theorem]{Lemma}

\newtheorem{proposition}[theorem]{Proposition}
}
\theoremstyle{definition}
{\swapnumbers
\newtheorem{definition}[theorem]{Definition}

\newtheorem{remark}[theorem]{Remark}
\newtheorem{paragr}[theorem]{}
}

\usepackage[margin=1in]{geometry}
\usepackage[final,
colorlinks=true,
linkcolor=RedViolet,
citecolor=RedViolet, 
urlcolor=RoyalBlue]{hyperref} 

\author{Andrew D. Lewis \and  Ahmed Gamal Shaltut}
\title{ Invariance of subbundles and nonholonomic
trajectories}
\date{}
\usepackage[backend=biber,
style=alphabetic,
autocite=inline
]{biblatex}
\newcommand\rtxt[1]{{#1}}

\begin{document}
\maketitle

\begin{abstract}
In the comparison of nonholonomic mechanics and constrained variational
mechanics, invariant affine subbundles arise in the determination of the
initial conditions where the two methods yield the same trajectories.
Motivated by this, differential conditions are considered for invariant
affine subbundles as they arise in the comparison of nonholonomic and
constrained variational mechanics.  First of all, the formal integrability of
the resulting linear partial differential equation is determined using
Spencer cohomology.  Second, iterative formulae are provided that permit the
determination of the largest invariant affine subbundle invariant under an
affine vector field.  Finally, the problem of a disc rolling on an inclined
plane with no-slip is considered as an example to illustrate the theory.
\begin{flushleft}
\textbf{AMS subject classifications (2020).} 35A01, 53B05, 58J90, 70F25,
70G45, 70G75.
\end{flushleft}
\end{abstract}

\section{Introduction} 

We warmly dedicate this paper to the memory of Professor Miguel
Mu\~noz-Lecanda.  Miguel was very kind to the first author early in his
career; his friendship was greatly valued.  The subject of this paper,
nonholonomic mechanics, was one in which Professor Mu\~noz-Lecanda made a
number of
contributions,~e.g.,~\autocite{MdL/ML/MCM-L:21,XG/JM-S/MCM-L:03,MCM-L:18}\@.
Unlike mechanical systems with holonomic constraints, constrained variational
trajectories for a system with nonholonomic constraints are not necessarily
the ones obtained from nonholonomic mechanics.  While the former has the
appeal of a variational framework, the latter capture the physics
correctly~\autocite{ADL:17b}\@. The case when constrained variational
trajectories are nonholonomic trajectories has been considered extensively in
the literature, for example in \autocite{OEF/AMB:08}, \autocite{MJ/WR:18},
\autocite{ADL:20a}, and \autocite{ADL/RMM:95a}.  As is well-known, the two
approaches are equivalent\textemdash{}in the sense that they \emph{always}
yield the same physical trajectories\textemdash{}if and only if the
constraint is holonomic.  The interesting question, and the one we are
concerned with in this paper, is when \emph{some} nonholonomic trajectories
arise from \emph{some} constrained variational trajectory.  For a review of
the history of the subject, see \autocite{MdL:12}.

In the recent exposition of Lewis~\autocite{ADL:20a} using an affine
connection formulation for the equations of the extermals, the two sets of
trajectories were compared.  Although restricted to kinetic energy minus
potential energy Lagrangians, previous results for the classification of
regular constrained variational trajectories as nonholonomic were recovered
and, for the first time, a classification of singular constrained variational
trajectories as nonholonomic trajectories was presented.  \rtxt{For a
	detailed discussion of how these results relate to existing work, including
	that cited above, we refer the reader to the discussion in \S7.3 of
	\autocite{ADL:20a}\@.}

The conditions arrived at in~\autocite{ADL:20a} are algebraic and
differential in nature and are for the existence of a flow-invariant affine
subbundle variety contained in a cogeneralized subbundle for the regular case
and of a flow-invariant cogeneralized subbundle contained in a cogeneralized
subbundle in the singular case (see Sections~\ref{subsec:cogen}
and~\ref{subsec:affvars} for the definitions).  Our objectives are to explore
these conditions more deeply, and to understand how to apply them by
considering a fairly simple, but yet illustrative, example.

The following is an outline of the paper.
\begin{enumerate}[1.,nosep]
	\item We fix our notation and conventions in Section~\ref{sec:prelim}.  In
	this section, we also review the relevant definitions and results
	from~\autocite{ADL:20a}\@.
	\item In Section~\ref{sec:mechanics} we present the equations governing
	nonholonomic and constrained variational trajectories, following the
	refinements of~\autocite{ADL:20a} of the initial work
	in~\autocite{IK/WMO:01}\@.  In this section we also indicate how the notions
	of invariant subbundles from Sections~\ref{subsec:cogen}
	and~\ref{subsec:affvars} arise in the comparison of nonholonomic and
	constrained variational trajectories.
	\item In Section~\ref{sec:invariance} we explore two aspects of the results
	of~\autocite{ADL:20a}\@: (a)~the formal integrability of differential
	conditions for invariance; (b)~the determination of computable infinitesimal
	characterisations of invariance.
	\item We provide a complete geometric formulation of the problem of a disc
	rolling with no-slip over an inclined plane as suggested by \autocite{NAL:22}
	in Section~\ref{sec:example}.  For this example, we carry out a detailed
	analysis of this example \emph{vis-\`a-vis} the question of comparing
	nonholonomic and constrained variational mechanics.  As we shall see, our
	results allow for a \emph{complete} characterization of the existence of
	invariant affine subbundle varieties, something which is typically not done
	due to implicit assumptions about the absence of singularities.
\end{enumerate}   

\section{Preliminaries}\label{sec:prelim}  

In this section we review some elementary geometric constructions for the
purpose of fixing notation, and we review the relevant definitions and
results concerning subbundles and their invariance under flows of certain
sorts of vector fields. \rtxt{For reasons of space, we present these results
	without the proofs and refer the reader to \autocite{ADL:20a} for the
	complete exposition and the detailed proofs.}

\subsection{Notation and elementary constructions}

To treat both smooth and real-analytic regularities, we let
$r\in\{\infty,\omega\}$.
\begin{paragr}
	Given a linear connection $\nabla$ on a vector bundle $\pi\colon\E\to\M$, the
	tangent bundle $\TE$ can be decomposed into horizontal and vertical
	bundles. These are denoted by $\HE$ and $\VE$, respectively
	\autocite[\S11.10]{kms}.  Recall that the
	horizontal lift isomorphism is given by the map $\hlft\colon\pi^*\TM\to\HE$
	such that, for $e\in \E$,
	\begin{equation*}
		T\pi(\hlft(e,v))=v,\quad v\in\mathsf{T}_{\pi(e)}\M.
	\end{equation*}
	The vertical lift isomorphism is the map $\vlft:\E\oplus\E\to\VE$ such
	that, for $e\in\E$,
	\begin{equation*}
		\vlft(e,f)=\left.\ddt\right|_{t=0}(e+tf),\quad f\in\E_{\pi(e)}.
	\end{equation*}
	Let $X_0\in\Gamma^r(\TM)$.  The horizontal lift of $X_0$ is the vector field
	$X_0^{\mathrm{h}}\in\Gamma^r(\TE)$ defined by
	\begin{equation*}
		X_0^{\mathrm{h}}(e)=\hlft(e,X_0(\pi(e))),\quad e\in\E.
	\end{equation*}
	The vertical evaluation of $A\in\Gamma^r(\mathrm{End}(\E))$ is the vector
	field $A^{\mathrm{e}}\in\Gamma^r(\TE)$ defined by
	\begin{equation*}
		A^{\mathrm{e}}(e) = \vlft(e,A(\pi(e))(e)),\quad e\in\E.
	\end{equation*}
	The horizontal lift of $f\in\C^r(\M)$ is the pullback using $\pi$, that is,
	$f^{\mathrm{h}}=\pi^*f\in\C^r(\E)$.  The vertical evaluation of
	$\lambda\in\Gamma^r(\E^*)$ is the function $\lambda^{\mathrm{e}}\in\C^r(\E)$
	defined by
	\begin{equation*}
		\lambda^{\mathrm{e}}(e)=\langle\lambda(\pi(e)),e\rangle,\quad e\in\E.
	\end{equation*}
	\rtxt{The vertical lift of any section $b\in \Gamma^r(\E)$ is defined by
		\begin{equation*}
			b^{\mathrm{v}} (e) = \vlft(e, b(\pi(e))),\quad e\in \E.
	\end{equation*}}
\end{paragr}

\begin{paragr}
	A vector field $X\in\Gamma^r(\TE)$ is called a linear (resp.\ affine) vector
	field over $X_0$ if it is a $\C^r$-vector (resp.\ affine) bundle map over
	$X_0$ such that the following diagram commutes
	\begin{equation*}
		\begin{tikzcd}
			\E\arrow{r}{X}\arrow[swap]{d}{\pi}&\TE\arrow{d}{T\pi}\\
			\M\arrow{r}[swap]{X_0}&\TM
		\end{tikzcd}
	\end{equation*} 
	If $X\in\Gamma^r(\TE)$ is an affine vector field over $X_0$, then, using the
	linear connection $\nabla$ in $\E$, it can be decomposed as
	\begin{equation*}
		X=X_0^{\mathrm{h}}+A^{\mathrm{e}}+b^{\mathrm{v}},
	\end{equation*}
	where $A\in\Gamma^r(\mathrm{End}(\E))$, and $b\in\Gamma^r(\E)$. 
	The case
	where $b$ is the zero section corresponds to linear vector fields. It is easy
	to see that a smooth curve $\eta\colon I\to \E$ is an integral curve of
	$X_0^{\mathrm{h}}+A^{\mathrm{e}}+b^{\mathrm{v}}$ if and only if
	\begin{equation}
		\nabla_{\gamma'}\eta = A\circ\eta + b\circ\gamma,
	\end{equation}
	where $\gamma=\pi\circ\eta$. 
\end{paragr}

\begin{paragr}
	The flow of a vector field $Y\in\Gamma^r(\TE)$ at $e\in\E$ is denoted by
	$\Fl^Y_t(e)$.  The dual vector field of $Y$ on the dual bundle $\E^*$, is
	the vector field $Y^*$ on $\E^*$ defined by
	\begin{equation*}
		Y^*(\lambda)=\left.\ddt\right|_{t=0}(\Fl^Y_{-t})^*(\lambda),\quad \lambda\in\E^*.
	\end{equation*}
	If $Y$ is a linear vector field over $X_0$ then $Y^*$ is linear vector field
	over $X_0$. Moreover, if $Y=X_0^{\mathrm{h}}+A^{\mathrm{e}}$, then
	$Y^*=X_0^{\mathrm{h},*} - (A^*)^{\mathrm{e}}$, where $X_0^{\mathrm{h},*} $ is
	the horizontal lift of $X_0$ corresponding to the dual connection on $\E^*$.
\end{paragr}

We need a few elementary results regarding Lie differentiation. We record
them in the following elementary lemma. See \autocite{ADL:20a} for the proof.
\begin{lemma}\label{lem:itrLieDeriv}
	Let $X_0\in\Gamma^r(\TM)$, let $A\in\Gamma^r(\mathrm{End}(\E))$, and let
	$b\in\Gamma^r(\E)$.  Then, for $k\in\mathbb{Z}_{>0}$, for $U\subseteq\M$
	open, for $f\in\C^r(U)$, and for any local section $\lambda$ of $\E^*$ over
	$U$, the following holds:
	\begin{enumerate}[(i),itemsep=0.1em]
		\item $\scrLXh^k f^{\mathrm{h}} = (\scrLXk f)^{\mathrm{h}}$;
		\item \label{item:Xafffh}
		$\scrLAe^kf^{\mathrm{h}}= \scrLbv^kf^{\mathrm{h}}=0$;
		\item \label{item:Xhlambda}
		$\scrLXh^k\lambda^{\mathrm{e}}=\left(\nx^k\lambda\right)^{\mathrm{e}}$;
		\item \label{item:Aelambda} $\scrLAe^k\lambda^{\mathrm{e}}=\left((A^*)^k\lambda\right)^{\mathrm{e}}$;
		\item \label{item:bvlambda}
		$\scrLbv\lambda^{\mathrm{e}}=\langle\lambda,b\rangle^{\mathrm{h}}$.
	\end{enumerate}
\end{lemma}

\begin{paragr}
	Let $\bbG$ be a $\C^r$-Riemannian metric on $\M$\@. Then $\bbG$ defines a
	musical isomorphism, the flat isomorphism
	$\bbG^\flat\colon\TM\to\T^*\M$. This is defined by
	$\bbG^\flat(v_x)(u_x)=\bbG(u_x,v_x)$ for $u_x,v_x\in\T_x\M$.  The inverse of
	$\bbG^\flat$ is denoted by $\bbG^\sharp$\@.  We denote by
	$\stackrel{\mathclap{ \scriptscriptstyle \G}}{\nabla}$ the Levi-Civita
	connection associated with $\bbG$\@.
\end{paragr}

\begin{paragr}\label{par:affine-distribution}
	Let $\D \subseteq \TM$ be a $\C^r$-distribution, i.e., a $\C^r$-subbundle of
	$\TM$.  If $\mathbb{G}$ is a $\C^r$-Riemannian metric on $\M$, we denote the
	$\G$-orthogonal complement of $\D$ by $\D^\perp$.  The {$ \G$-orthogonal
		projections} are given by the vector bundle maps $P_\D\colon\TM\to\D$ and
	$P_{\D^\perp}\colon\TM\to\D^\perp$ projecting onto $\D$ and $\D^\perp$,
	respectively.  The constrained connection is the affine connection
	$\stackrel{\mathclap {\scriptscriptstyle \D}}{\nabla}$ on $\M$ that is
	defined by
	\begin{equation*}
		\stackrel{\mathclap {\scriptscriptstyle \D}}{\nabla}
		_XY=\stackrel{\mathclap{ \scriptscriptstyle \G}}{\nabla}
		_XY + (\stackrel{\mathclap {\scriptscriptstyle \G}}{\nabla}_XP_{\D^\perp}) Y,\quad X,Y\in\Gamma^r(\TM).
	\end{equation*}
	It follows that $\nabd_X Y = P_\D \circ \nabg_X Y$.  The second fundamental
	form for $\D$ is the tensor field
	$S_\D\in \Gamma^r(\TsM\otimes \D^*\otimes \D^\perp)$ defined by
	\begin{equation*}
		S_\D(X,Y)=-(\stackrel{\mathclap {\scriptscriptstyle \G}}{\nabla}_XP_{\D^\perp}) Y,\quad X\in\Gamma^r(\TM),\ 
		Y \in \Gamma^r(\D).
	\end{equation*}
	\rtxt{Using $S_\D$, we define a symmetric and an alternating tensor fields on
		$\D^*$ taking values in $\D^\perp$.  First, is the {geodesic curvature}
		$G_\D\in\Gamma^r( S^2(\D^*)\otimes \D^\perp)$ which is defined by
		\begin{equation*}
			G_\D(X,Y)=S_\D(X,Y)+S_\D(Y,X)\quad X,Y\in\Gamma^r(\D),
		\end{equation*}
		and second is the Frobenius curvature 
		$F_\D\in\Gamma^r(\bigwedge^2(\D^*)\otimes \D^\perp)$ which is defined by
		\begin{equation*}
			F_\D(X,Y)=S_\D(X,Y)-S_\D(Y,X)\quad X,Y\in\Gamma^r(\D).
	\end{equation*}}
	
	Related to the above definitions of $F_\D$ and $G_\D$, we consider their
	transposes.  We define $F^\star_\D,G^\star_\D\in \Gamma^r(\D\otimes\D^*)$
	such that, for $\beta\in\Gamma^r(\D^\perp)$ and $X,Y\in\Gamma^r(\D)$,
	\begin{equation*}
		\G(X,F^\star_\D(\beta)(Y))=\G(\beta,F_\D(X,Y)),\quad
		\G(X,G^\star_\D(\beta)(Y))=\G(\beta,G_\D(X,Y))
	\end{equation*}
	holds. Similarly, we define the maps $F^*_\D$
	and $G^*_\D$:
	\begin{equation*}
		\G(X,F^*_\D(Y)(\beta))=\G(\beta,F_\D(X,Y)),\quad
		\G(X,G^*_\D(Y)(\beta))=\G(\beta,G_\D(X,Y)).
	\end{equation*}
	This can be seen as flipping of arguments, that is
	\begin{equation*}
		F^\star_\D(\beta)(Y)=F^*_\D(Y)(\beta),\quad
		G^\star_\D(\beta)(Y)=G^*_\D(Y)(\beta).
	\end{equation*}
\end{paragr}

\subsection{(Co)generalized (affine) subbundles and their
	invariance}\label{subsec:cogen}

Let $\pi\colon \E\to \M$ be a $\C^r$-vector bundle.  We wish to consider
generalizations of subbundles for which the rank is allowed to vary.  The
first sort of such a generalization is the following.
\begin{definition}
	A \emph{generalized subbundle} of $\E$ is a subset $\F \subseteq \E$ such
	that, for every $x\in \M$, there exists an open neighborhood $U$ of $x$ and a
	collection of sections $\{\xi_i\}_{i\in I} \subseteq \Gamma^r(\pi^{-1}(U))$
	such that
	\begin{equation*}
		\pi^{-1}(y)\cap \F = \mathrm{span}_\mathbb{R} \left\{\xi_i(y)\ | \ i \in I \right\}, \quad y \in U.
	\end{equation*}
\end{definition}

The local sections are called the \emph{local generators} of $\F$.  Without
loss of generality, these can be taken to be a finite collection of global
sections of $\E$ \autocite[Corollary~2.18]{ADL:20a}.  For each $x\in \M$, the
subspace $\pi^{-1}(x)\cap \F$ is denoted by $\F_x$ and its annihilator by
$\Lambda(\F)_x\subseteq \E^*_x$. We denote the union of all annihilators by
$\Lambda(\F)\subseteq \E^*$, that is, $\E^*_x\cap
\Lambda(\F)=\Lambda(\F)_x$. For any subset $S\subseteq \M$, the set
$\F\cap \pi^{-1}(S)$ is denoted by $\F|S$.

Using annihilators, we can define another sort of a generalization of a subbundle.
\begin{definition}
	A \emph{cogeneralized subbundle} is a subset $\F \subseteq \E$ such that
	$\Lambda(\F)$ is generalized subbundle of $\E^*$\@.
\end{definition}

For a (co)generalized subbundle $\F\subset\E$\@, a point $x\in \M$ is a
\emph{regular point} of $\F$ if there exists a neighborhood $U$ of $x$ such
that $\F|U$ has constant rank. A \emph{singular point} of $\F$ is a point
which is not a regular point.
\begin{remark}\label{rem:dense}
	For any generalized or cogeneralized subbundle $\F$, there exists an open
	dense set $U\subseteq \M$ consisting of regular points of $\F$
	\autocite[Lemma~2.20]{ADL:20a}. That is, $\F|U$ is a subbundle of $\E|U$. In
	particular, all points are regular for subbundles of $\E$.
\end{remark}

As we will be working with affine vector fields, the notion of an invariant
subbundle is of less interest than that of an invariant \emph{affine}
subbundle.  Thus we need to define what we mean by an affine subbundle.
\begin{definition}
	A \emph{generalized affine subbundle} is a subset $\B\subset\E$ for which
	there exists a generalized subbundle $\F\subset\E$ and a section
	$\xi_0\in\Gamma^r(\TM)$ such that
	\begin{equation*}
		\B_x:=\B\cap\E_x=\xi_0(x)+\F_x.
	\end{equation*}
\end{definition}

One can phrase the definition of a generalized affine subbundle in terms of
generators~\autocite[Lemma~2.23]{ADL:20a}\@.

There is also a cogeneralized version of an affine subbundle.
\begin{definition}
	A \emph{cogeneralized affine subbundle} is a subset $\B\subset\E$ for which
	there exists a cogeneralized subbundle $\F\subset\E$ and a section
	$\xi_0\in\Gamma^r(\TM)$ such that
	\begin{equation*}
		\B_x:=\B\cap\E_x=\xi_0(x)+\F_x.
	\end{equation*}
\end{definition}

Now we give the definitions for invariance of (co)generalized (affine)
subbundles under affine vector fields.  In \autocite{ADL:20a}\@, two types of
invariance are discussed, ``flow-invariance'' and ``invariance,'' and much of
the work of the paper is devoted to giving the relationships between these.
The notion of ``flow-invariance'' is exactly what one would expect: \emph{a set is
	{flow-invariant} under a vector field if the integral curves through
	points in the set remain in the set.}  The notion of ``invariance'' is more
delicate, and involves the ideal sheaf of the set,~i.e.,~the sheaf of
functions that vanish on the set.  Our interest is in (flow-)invariant sets
that are (co)generalized (affine) subbundles.  To this end, we make two
comments.
\begin{enumerate}[1.,nosep]
	\item For subbundles and affine subbundles, the ideal sheaf can be
	characterized by particular classes of functions, namely linear and affine
	functions, respectively.
	\item The ideal sheaf is only meaningful for cogeneralized (affine)
	subbundles because these are naturally defined by the vanishing of suitable
	classes of functions.
\end{enumerate}
We refer to \autocite[\S4]{ADL:20a} for a detailed discussion of this.  What
we shall do here is give the results that characterize the invariant objects
that are of interest to us here.  To do so, we denote by $\calC^r_{\M}$ the
sheaf of $\C^r$-functions on $\M$ and by $\calG^r_{\F}$ the sheaf of sections
of a generalized subbundle $\F$\@.

The following is a characterization of flow-invariance in terms of algebraic
and differential conditions. See \autocite[\S4]{ADL:20a} for the details.
\begin{proposition}\label{prop:cogenflow}
	Let $\B=\xi_0+\F$ be a cogeneralized affine subbundle. If $\B$ is
	flow-invariant under
	$\Xaff=X_0^{\mathrm{h}} + A^{\mathrm{e}} + b^{\mathrm{v}}$\@, then the
	following conditions hold:
	\begin{enumerate}[(i),nosep]
		\item $A(x) \in \End(\F_x)$ for all $x\in \M$\@;
		\item \label{pl:cogenflow2} $\nabla_{X_0} (\calG^r_{\Lambda(\F)}) \subseteq\calG^r_{\Lambda(\F)}$\@;
		\item $(\nabla_{X_0}\xi_0-A\circ\xi_0-b)(x)\in\F_x$\@,\/ $x\in\M$\@.
	\end{enumerate}
	The converse is true in the real analytic case and if\/ $\F$ is a subbundle
	in the smooth case.
\end{proposition}

The same result holds for generalized subbundles, but the differential
condition is on $\F$ not its annihilator. That is, the second condition
$\nabla_{X_0} (\calG^r_{\F}) \subseteq\calG^r_{\F}$.  Note that the
conditions of the proposition are differential/algebraic in nature.  One of
the goals of this paper is to show that the differential part of the
result defines a linear partial differential equation that is formally
integrable under certain hypotheses.

\subsection{Affine subbundle varieties and their
	invariance}\label{subsec:affvars}

The essential problem of interest in \autocite{ADL:20a} is the following:
given a cogeneralized subbundle $\F\subset\E$ and an affine vector field
$\Xaff$ on $\E$\@, find a cogeneralized affine subbundle $\B$ that is
(a)~invariant under $\Xaff$ and (b)~contained in $\F$\@.  The difficulty with
this sort of problem is that the affine subbundle may be empty, or have
nonempty fibers over a strict subset of $\M$\@.  To facilitate the analysis
in this case, it turns out to be most convenient to recast the problem from
one of solutions to linear equations to the linear equations themselves.  In
this way, one can talk about the equation, even when it has no solutions.

We begin with the algebraic setting.

\begin{paragr}
	Let $\V$ be a finite-dimensional $\mathbb{R}$-vector space, let
	$A\in \mathrm{End}(\V)$, and let $b\in\V$. Consider the set
	\begin{equation*}
		\mbox{Sol}(A,b) = \left\{ v\in \V\ |\  A(v)+b =0 \right\}.
	\end{equation*}
	Using the dual $A^*$, we define the subspace
	\begin{equation*}
		\mbox{Sol}^*(A,b) = \left\{   
		(A^*(\lambda),\langle \lambda, b\rangle) \in \V^*\oplus\mathbb{R}
		\ |\ \lambda \in \V^*\right\}.
	\end{equation*}
	This subspace has a positive codimension in $\V^*\oplus\mathbb{R}$.
	Alternatively, if $\Delta \subseteq\V^*\oplus\mathbb{R}$ is a subspace of
	positive codimension, then $\Delta=\mbox{Sol}^*(A,b)$ for some
	$A\in\mbox{End}(\V)$ and $b\in \V$. Furthermore, one can easily see that
	\begin{equation*}
		\mbox{Sol}(A,b) = \{v\in\V\ |\ (v,1)\in\Lambda(\Delta)\}.
	\end{equation*}
	We remark that the set $\mbox{Sol}(A,b)$ uniquely defines $\Delta$, but does
	not uniquely define $A$ and $b$ \autocite[Lemma~2.29]{ADL:20a}.
\end{paragr}

Now we adapt this to the geometric setting.  By $\mathbb{R}_{\M}$ we denote
the trivial line bundle over $\M$\@.
\begin{definition}
	\begin{enumerate}[(i),nosep]
		\item A generalized subbundle $\Delta \subseteq \E^*\oplus \mathbb{R}_\M$ is
		a \emph {defining subbundle} if $\Delta_x$ has positive codimension in
		$\E^*_x\oplus\mathbb{R}$ for all $x\in\M$.
		
		\item A subset $\A \subseteq \E$ is called an \emph{affine subbundle variety}
		if there exists a defining subbundle $\Delta$ such that
		\begin{equation*}
			\A=\A(\Delta):=\{ e\in\E\ |\ \lambda(e)+a=0,\  (\lambda,a)\in
			\Delta_{\pi(e)} \}.
		\end{equation*} 
		
		\item The subset 
		\begin{equation*}
			\mathsf{S}(\A)=\{x\in\M\ |\ \A\cap\E_x\neq \emptyset \}
		\end{equation*}
		is called the \emph{base variety} of $\A$.
	\end{enumerate}
	A defining subbundle $\Delta$ is \emph{total} if $\sS(\A(\Delta))=\M$,
	\emph{partial} if $\emptyset\neq\sS(\A(\Delta))\subsetneq \M$, and
	\emph{null} if $\sS(\A(\Delta))=\emptyset$.
\end{definition}

The idea of the terminology of total (resp.\ partial, null) is that the
linear equation defined by $\Delta$ has solutions for all (resp.\ some, no)
points in $\M$\@.

Suppose that $\A$ is an affine subbundle variety. Let $\A_x$ denote the
affine subspace $\A\cap\E_x$, for each $x\in\mathsf{S}(\A)$.  It follows that
there exists a unique subspace $\Delta_x \subseteq\E^*_x\oplus\mathbb{R}$
such that $\A_x=\{ e\in\E_x\ |\ (e,1)\in\Lambda(\Delta_x) \}$.  Let
$\Delta_{1,x}$ denote the image of $\Delta_x$ under the projection
\begin{equation*}
	\E_x^*\oplus\mathbb{R} \to (\E_x^*\oplus\mathbb{R})/(\{0\}\oplus\mathbb{R})
	\simeq \E_x^*,
\end{equation*} 
and set $\Delta_1=\cup_{x\in\M}\Delta_{1,x}$.

\begin{paragr}
	Before \rtxt{characterizing} the flow-invariance of defining subbundles and
	affine subbundle varieties, we replace the affine vector field $\Xaff$ on
	$\E$ with a linear vector field $\what{\Xaff}$ on $\E\oplus\bbR_\M$.  Let
	$\Xaff=X_0^{\mathrm{h}} + A^{\mathrm{e}}+ b^{\mathrm{v}}$ be an affine vector
	field over $X_0$ and set $\Xlin = X_0^{\mathrm{h}} + A^{\mathrm{e}}$.  Let
	$\what{\nabla}$ be the connection on $\E\oplus\mathbb{R}_\M$ obtained by
	taking the direct sum of $\nabla$ in $\E$ and the flat connection in
	$\mathbb{R}_\M$. That is
	\begin{equation*}
		\what{\nabla}_X(\xi,f)=(\nabla_X\xi,\mathscr{L}_Xf),\quad
		X\in\Gamma^r(\TM),\ (\xi,f)\in\Gamma^r(\E\oplus\mathbb{R}_\M).
	\end{equation*}
	We define a linear vector field on $\E\oplus\mathbb{R}_\E$ by
	\begin{equation*}
		\what{\Xaff} = X_0^{\mathrm{h}} + \what{(A,b)}^{\mathrm{e}},
	\end{equation*}
	where the horizontal lift of $X_0$ is done using $\what{\nabla}$, and
	$\what{(A,b)}\in\Gamma^r(\mathrm{End}(\E\oplus\mathbb{R}_\M))$ is defined
	by
	\begin{equation*}
		\what{(A,b)}(\xi,f) = (A\xi+fb, 0),\quad
		(\xi,f)\in\Gamma^r(\E\oplus\mathbb{R}_\M).
	\end{equation*}
	We note that the dual of this linear vector field is
	\begin{equation*}
		\what{\Xaff}^* = X_0^{\mathrm{h}} - \what{(A,b)}^{*,\mathrm{e}}.
	\end{equation*}
	The importance of the vector field $\what{\Xaff}$ is explained by the
	following result~\cite[Lemma~4.18]{ADL:20a}\@.
	\begin{proposition}
		Let\/ $\Delta$ be a defining subbundle with\/ $\A$ the associated affine
		subbundle variety.  Then the following statements are equivalent:
		\begin{enumerate}[(i),nosep]
			\item $\A\subset\E$ is flow-invariant under\/ $\Xaff$\@;
			\item $\{(e,1)\in\E\times\mathbb{R}_\M\ |\ e\in\A\}$ is flow-invariant
			under\/ $\what{\Xaff}$\@;
			\item $\Delta$ is flow-invariant under\/ $\what{\Xaff}^*$\@.
		\end{enumerate}
	\end{proposition}
\end{paragr}

The previous result allows us to replace the quite subtle notion of
invariance of an affine subbundle variety under an affine vector field with
the notion of invariance of an affine subbundle variety under a linear vector
field on an extended bundle, and then to replace this with the notion of
invariance of a subbundle on the dual bundle of the extended bundle under
a linear vector field.  By working with $\Delta$ rather than $\A(\Delta)$\@,
one does not have to concern oneself with whether $\A(\Delta)$ is nonempty
when discussing invariance.

\begin{paragr}
	The constructions from the proceeding paragraph allow for a convenient test
	for invariance of affine subbundle varieties.  As we mentioned at the
	beginning of this section, we are interested in invariant affine subbundle
	varieties that are contained in a cogeneralized subbundle $\F$\@.  To
	intertwine invariance and containment in $\F$\@, we introduce the following
	notation for a cogeneralized subbundle $\F$ and a defining subbundle
	$\Delta$\@:
	\begin{equation*}
		\what{\Lambda(\Delta)} = \Lambda(\Delta)\cap(\E\times\{1\}),\quad
		\what{\F}=\{(e,1)\in\E\oplus\mathbb{R}_\M\ |\ e\in\F\}.
	\end{equation*}
	It is easy to see that the affine subbundle variety associated with $\Delta$
	is
	\begin{equation*}
		\A=\{e\in\E\ |\ (e,1)\in\what{\Lambda(\Delta)}\}
	\end{equation*}
	and that $\A\subset\F$ if and only if
	$\what{\Lambda(\Delta)}\subset\what{\F}$\@.
	
	We adapt the following result from
	\autocite[Remark~4.26]{ADL:20a}.
	\begin{proposition}\label{prop:defnsub}
		Let\/ $\F$ be a cogeneralized subbundle and suppose that\/ $\Delta$ is a
		defining subbundle that is flow-invariant under\/
		$\Xaff\rtxt{=X_0^{\mathrm{h}} + A^{\mathrm{e}}+ b^{\mathrm{v}}}$. If
		$\what{\Lambda(\Delta)} \subseteq \what{\F}$, then the following statements
		hold:
		\begin{enumerate}[(i),nosep]
			\item $A(\Lambda(\Delta_{1,x}))  \subseteq \F_x$ for $x\in\M$;
			\item $\nabla_{X_0}(\mathcal G^r_{\Lambda(\F)}) \subseteq \mathcal{G}^r_{\Delta_1}$.
		\end{enumerate}
		If\/ $\what{\Lambda(\Delta)}\cap\what{\F}\not=\emptyset$\@, then the converse
		is true in the real analytic case and if $\F$ is a subbundle in the smooth
		case.
	\end{proposition}
	
	\begin{remark}
		In particular, if $\Delta$ and $\F$ satisfy the hypothesis of
		Proposition~\ref{prop:defnsub} (including the hypothesis that
		$\what{\Lambda(\Delta)}\cap\what{\F}\not=\emptyset$), then $\A(\Delta)$ is
		flow-invariant and contained in $\F$.
	\end{remark}
	
	Following the previous characterization, we have the following definition for
	defining subbundles for which the associated affine subbundle variety is
	flow-invariant and contained in $\F$.
	\begin{definition}\label{def:adm}
		A defining subbundle $\Delta$ is \emph{$(\Xaff,\F)$-linearly admissible} if
		it is flow-invariant under $\Xaff$ and the following conditions hold:
		\begin{enumerate}[(i),nosep]
			\item $A(\Lambda(\Delta_{1,x}))  \subseteq \F_x$ for $x\in\M$;
			\item $\nabla_{X_0}(\mathcal G^r_{\Lambda(\F)}) \subseteq \mathcal{G}^r_{\Delta_1}$.
		\end{enumerate}
		If, additionally,\/ $\what{\Lambda(\Delta)}\not=\emptyset$\@, then we say
		that $\Delta$ is \emph{$(\Xaff,\F)$-admissible}\@.
	\end{definition}
\end{paragr}

One of the objectives of the paper is to find a constructive means to find
flow-invariant affine subbundle varieties.


\section{Nonholonomic and constrained variational
	mechanics}\label{sec:mechanics}

In this section we give the governing differential equations for the two
types of mechanics we are comparing, nonholonomic mechanics and constrained
variational mechanics.  We only consider mechanical systems with ``kinetic
minus potential energy'' Lagrangians.  These are certainly of significant
interest, and also have enough structure to be able to draw deeper
conclusions that one will be able to draw with general Lagrangians.  The
equations we present here are an evolution presented in~\autocite{ADL:20a}\@,
based on the work in~\autocite{IK/WMO:01}\@.

We begin with the definition of the mechanical systems we consider.
\begin{definition}
	A \emph{$\C^r$-constrained (simple) mechanical system} is a quadruple
	$(\M,\G,V,\D)$ of a $\C^r$-manifold $\M$, representing the configuration
	manifold, a $\C^r$-Riemannian metric $\G$ representing the kinetic energy, a
	$\C^r$-potential function $V$ representing the potential energy, and a
	$\C^r$-distribution representing the constraints.
\end{definition}

Now we consider the two different equations of motion.  Note that, in the
case of constrained variational mechanics we have \emph{two} cases for the
equations of motion, corresponding to normal (called ``regular'' in the
definition) and abnormal (called ``singular'' in the definition) extremals.
The connection between these definitions and variational formulations is
established by~\autocite[Theorems~5.18 and~5.22]{ADL:20a}\@.
\begin{definition}
	Let $I\subset\mathbb{R}$ be an interval and let $\gamma\colon I\to\M$ be an
	absolutely continuous curve such that $\gamma'(t)\in\D_{\gamma(t)}$ a.e.\ on
	$I$\@.
	\begin{enumerate}[(i),nosep]
		\item The curve $\gamma$ is a \emph{nonholonomic trajectory} if it is $\C^1$
		with an absolutely continuous derivative and
		\begin{equation}\label{eq:nh}\tag{NH}
			\stackrel{\mathclap{\scriptscriptstyle{\D}}}{\nabla}_{\gamma'} \gamma'
			+ P_\D\circ \mathrm{grad}\ V \circ \gamma = 0.
		\end{equation}			
		
		\item The curve $\gamma$ is a \emph{$\D$-regular constrained variational
			trajectory} if it is $\C^1$ with an absolutely continuous derivative and if
		there exists an absolutely continuous curve $\lambda\colon I\to\D^\perp$ over
		$\gamma$ such that
		\begin{equation}\label{eq:dreg}\tag{RCV}
			\begin{split}
				&\stackrel{\mathclap{\scriptscriptstyle{\D}}}{\nabla}_{\gamma'} \gamma'
				+P_\D\circ \mathrm{grad}\ V \circ \gamma=F_\D^*(\gamma')(\lambda),\\
				&\nablaDns_{\gamma'}\lambda = 
				\frac{1}{2}G_\D(\gamma',\gamma')+ P_{\D^\perp}\circ \mathrm{grad}\ V \circ \gamma +
				\frac{1}{2}G^\star_{\D^\perp}(\gamma')
				(\lambda) + \frac{1}{2}F^\star_{\D^\perp}(\gamma')(\lambda).
			\end{split}
		\end{equation}
		
		\item The curve $\gamma$ is a \emph{$\D$-singular constrained variational
			trajectory} if there exists a nowhere zero, absolutely continuous curve
		$\lambda\colon I\to\D^\perp$ over $\gamma$ such that
		\begin{equation}\label{eq:dsing}\tag{SCV}
			\begin{split}
				&F_\D^*(\gamma')(\lambda)=0,\\
				&\nablaDns_{\gamma'}\lambda = \frac{1}{2}G^\star_{\D^\perp}(\gamma')
				(\lambda) + \frac{1}{2}F^\star_{\D^\perp}(\gamma')(\lambda).
			\end{split}
		\end{equation}
	\end{enumerate}
	The curve $\lambda$ over $\gamma$ is the \emph{adjoint field}\@.
\end{definition}

We are primarily interested in this paper with comparing nonholonomic
trajectories with $\D$-regular constrained variational trajectories.  Note
that the first of equations~\eqref{eq:dreg} agrees with the
equation~\eqref{eq:nh} on the left-hand side, while the former has a term
involving the Frobenius curvature on the right-hand side.  Essentially, we
wish to consider under what initial conditions for the second of
equations~\eqref{eq:dreg} will the term on the right in the first equation
vanish.  Indeed, in such cases, the corresponding $\D$-regular constrained
variational trajectory will also be a nonholonomic trajectory.

To develop the framework for this, we pull back the second of the
equations~\eqref{eq:dreg} to $\D$ to get a vector field over the vector field
defining the nonholonomic trajectories.  We let $\pi_\D\colon\D\to\M$ be the
restriction of the tangent bundle projection.  We let $\pi_\D^*\D$ and
$\pi_\D^*\D^\perp$ denote the pull-back bundles.  We use
$\pi_\D\colon\D\to\M$ to pull back the tensor fields $F^*_\D$,
$F^\star_{\D^\perp}$, and $G^\star_{\D^\perp}$ to get
\begin{equation*}
	\begin{split}
		\hat{F}^*_\D: &\pi_\D^*\D^\perp \to \pi_\D^*\D\\
		&(v_x,\alpha_x)\mapsto (v_x,F_\D^*(v_x)(\alpha_x)),\\
		\hat{F}^\star_{\D^\perp}: &\pi_\D^*\D^\perp \to \pi_\D^*\D^\perp\\
		&(v_x,\alpha_x)\mapsto (v_x,F_{\D^\perp}^\star(v_x)(\alpha_x)),\\
		\hat{G}^\star_{\D^\perp}: &\pi_\D^*\D^\perp \to \pi_\D^*\D^\perp\\
		&(v_x,\alpha_x)\mapsto (v_x,G_{\D^\perp}^\star(v_x)(\alpha_x)).
	\end{split}
\end{equation*}
We also pull back the connection
$\stackrel{\mathclap{\scriptscriptstyle{\D^\perp}}}{\nabla}$ to get the
connection $\nablaD$ in $ \pi_\D^*\pi_{\D^\perp}\colon\pi_\D^*\D^\perp \to\D$
defined by
\begin{equation*}
	\nablaD_\omega\pi_\D^*\alpha = (v, \nablaDns_{T_v\pi_\D(\omega)}\alpha),
	\quad v\in\D,\omega\in\T_v\D,\ \alpha\in\Gamma^r(\D^\perp).
\end{equation*}

Denote by $\Xnh\in\Gamma^r(\mathsf{T}\D)$ the vector field whose integral
curves are nonholonomic trajectories.  Define
$b_\D\in\Gamma^r(\pi^*_\D\D^\perp)$ and
$A_\D\in\Gamma^r(\mathrm{End}(\pi^*_\D\D^\perp))$ by
\begin{equation}\label{eq:beq}
	b_\D(v_x) = \left( v_x,\frac{1}{2}G_\D(v_x,v_x) +
	P_{\D^\perp}\circ\mathrm{grad}\ V\right),\quad v_x\in\D,
\end{equation}
and
\begin{equation}\label{eq:Aeq}
	A_\D(v_x,\alpha_x) =
	\frac{1}{2}\hat{G}^\star_{D^\perp}(v_x,\alpha_x)
	+\frac{1}{2}\hat{F}^\star_{D^\perp}(v_x,\alpha_x),\quad 
	v_x\in\D, (v_x,\alpha_x)\in\pi^*_\D\D^\perp.
\end{equation}
Then set
\begin{equation}\label{eq:Xsingreg}
	\begin{aligned}
		X^{\mathrm{sing}}_\D =&\;(X^{\mathrm{nh}}_\D)^{\mathrm{h}}+A_\D^{\mathrm{e}},\\
		X^{\mathrm{reg}}_\D=&\;(X^{\mathrm{nh}}_\D)^{\mathrm{h}}+A_\D^{\mathrm{e}}+
		b_\D^{\mathrm{v}},
	\end{aligned}
\end{equation}
where $(X^{\mathrm{nh}}_\D)^{\mathrm{h}}\in\Gamma^r(\T(\pi^*_\D\D^\perp))$ is
the horizontal lift of $X^{\mathrm{nh}}_\D$ by $\nablaD$.

Suppose that we have a curve $\gamma$ satisfying (\ref{eq:nh}) and let
$\Upsilon=\gamma'$.  Let $\hat{\lambda}\colon I\to \pi^*_\D\D^\perp$ be a
curve over $\Upsilon$ and write $\hat{\lambda}=(\Upsilon,\lambda)$.  The
second of the equations in~\eqref{eq:dreg} and~\eqref{eq:dsing} for the
adjoint field $\lambda$ can be written as
\begin{equation}\label{eq:affsys}
	\nablaD_{\Upsilon'}\hat{\lambda} = A_\D\circ\hat{\lambda} + b_\D\circ\Upsilon
	\quad\mbox{and}\quad
	\nablaD_{\Upsilon'}\hat{\lambda} = A_\D\circ\hat{\lambda},
\end{equation}
respectively.  Together with~\eqref{eq:nh}, these equations are equivalent to
$\hat{\lambda}$ being an integral curve for $X^{\mathrm{sing}}_\D$ or
$X^{\mathrm{reg}}_\D$, respectively, that projects to the nonholonomic
trajectory $\gamma$\@.  By examining equations~\eqref{eq:affsys}, we see that
they resemble an affine and a linear differential equation, respectively.  We
can also make the following observations.
\begin{enumerate}[1.,nosep]
	\item If $\hat{\lambda}(t)\in \ker(\hat{F}^*_\D)_{\Upsilon(t)}$ \rtxt{(the fiber 
		of $\ker(\hat{F}^*_\D)$ at $\Upsilon(t)$)} and
	$\hat{\lambda}$ is an integral curve of $X_\D^{\mathrm{reg}}$\@, then
	$\Upsilon$ is an integral curve of $X_\D^{\mathrm{nh}}$, and $\gamma$ and
	$\lambda$ satisfy~\eqref{eq:dreg}\@.
	\item If $\lambda(t)\neq 0$ for all $t\in I$,
	$\hat{\lambda}(t)\in\ker(\hat{F}^*_\D)_{\Upsilon(t)}$, and $\hat{\lambda}$ is
	an integral curve of $X_\D^{\mathrm{sing}}$ then $\Upsilon$ is an integral
	curve of $X_\D^{\mathrm{nh}}$, and $\gamma$ and $\lambda$ satisfy~\eqref{eq:dsing}.
\end{enumerate}

The following two results align with these
observations~\autocite[Theorems~7.8 and~7.9]{ADL:20a}. Indeed, we are looking
for flow-invariant cogeneralized subbundles and affine subbundle varieties
which are contained in $\ker(\hat{F}^*_\D)$\@.
\begin{theorem}\label{the:reg}
	Suppose that either (a)~$r=\omega$ and\/ $X_\D^{\mathrm{nh}}$ is complete or
	(b)~$r=\infty$ and\/ $\ker(\hat{F}_\D^*)$ is a subbundle. Then the following
	statements are equivalent:
	\begin{enumerate}[(i),nosep]
		\item some (resp.~all) nonholonomic trajectories are $\D$-regular constrained
		variational trajectories;
		\item there exists a partial (resp.~total)
		$(X^{\mathrm{reg}}_\D, \ker(\hat{F}^*_\D))$-admissible $\C^r$-defining
		subbundle $\Delta \subseteq (\pi^*_\D\D^\perp)^*\oplus\mathbb{R}_\D$.
	\end{enumerate}
\end{theorem}

\begin{theorem}\label{the:sing}
	Suppose that either (a)~$r=\omega$ and\/ $X_\D^{\mathrm{nh}}$ is complete or
	(b)~$r=\infty$ and\/ $\ker(\hat{F}_\D^*)$ is a subbundle.  Then the following
	are equivalent:
	\begin{enumerate}[(i),nosep]
		\item some nonholonomic trajectories are $\D$-singular constrained
		variational trajectories;
		\item there exists a flow-invariant cogeneralized subbundle under
		$X^{\mathrm{sing}}_\D$ that is contained in $\ker(\hat{F}^*_\D)$\@.
	\end{enumerate}
\end{theorem}

Together with the (flow-)invariance results stated in
Sections~\ref{subsec:cogen} and~\ref{subsec:affvars}\@, these results are of
fundamental importance in understanding questions of when constrained
variational trajectories are also nonholonomic trajectories.

\section{The differential conditions for invariance of (co)generalized
	subbundles}\label{sec:invariance}

As we remarked after the statement of Proposition~\ref{prop:cogenflow}\@, and
as can also be seen in Proposition~\ref{prop:defnsub}\@, the conditions for
flow-invariance of an affine subbundle variety under an affine vector field
involve differential and algebraic conditions. \rtxt{
	The first
	result of this section is concerned with the possibility that, upon differentiation of the
	differential condition, further algebraic conditions may reveal themselves.
	We give conditions under which this is not the case by using the theory of
	linear overdetermined partial differential equations developed
	in~\autocite{HLG:67a} and presented in~\autocite{JFP:78}\@.  The second
	result provides a concrete methodology for constructing an invariant affine
	subbundle variety that should prove useful in applications of the general theory.}

\subsection{Formal integrability}

We let $\pi\colon\E\to\M$ be a $\C^r$-vector bundle and let $\F\subset\E$ be
a subbundle of $\E$\@.  We let $\nabla$ be a linear connection on $\E$ and
consider a linear vector field $\Xlin=X_0^{\mathrm{h}}+A^{\mathrm{e}}$ over
$X_0\in\Gamma^r(\TM)$\@.  We are concerned with the analogue for generalized
subbundles of condition~\ref{pl:cogenflow2} in
Proposition~\ref{prop:cogenflow}\@.  That is to say, we are concerned with
the condition $\nabla_{X_0} (\calG^r_{\F}) \subseteq\calG^r_{\F}$\@.  We let
$P_{\E/\F}\colon\E\to\E/\F$ be the canonical projection, noting that $\E/\F$
is a vector bundle as we are assuming that $\F$ is a subbundle.  (We can just
as well use the projection onto a complement of $\F$ in $\E$\@.)  In this
case, the condition $\nabla_{X_0} (\calG^r_{\F}) \subseteq\calG^r_{\F}$ can
be written as $P_{\F^\perp}\circ\nabla_{X_0}\xi=0$ for
$\xi\in\Gamma^r(\F)$\@.

The next result gives two situations in which we have formal integrability of
this linear partial differential equation in the real analytic case.
\rtxt{For background on formal integrability of PDE, we refer the reader to
	\autocite[Section~4.2]{JFP:78} and \autocite{HLG:67a}.}

\begin{theorem}
	Let $\pi:\E\to\M$ be a real analytic vector bundle with a real analytic
	linear connection $\nabla$ and fiber-wise inner product, let
	$X_0\in{\Gamma^{\omega}({\TM})}$, let $\F\subseteq\E$ be a\/
	$\C^\omega$-subbundle, and let $ \F^\perp$ denote its (orthogonal) complement
	subbundle. Denote the canonical projection by
	$P_{\F^\perp}\colon\E\to \F^\perp$\@, and define a vector bundle map
	$\Phi\colon\mathsf{J}_1\E\to \F^\perp$ by
	$\Phi(j_1\xi(x)) = P_{\F^\perp}(\nabla_{X_0}(\xi(x)))$.  Suppose that
	$\ker(\Phi)$ is a subbundle.  Then $P_{\F^\perp}\circ\nabla_{X_0}$ is
	formally integrable if either of the following conditions holds:
	\begin{enumerate}[(i),nosep]
		\item \label{pl:spencer1} $X_0$ is nowhere vanishing;
		\item \label{pl:spencer2}
		$P_{\F^\perp}\circ\nabla\circ P_{\F^\perp}\circ\nabla \xi \in S^2\TsM \otimes
		\F^\perp$ for every $\xi\in\Gamma^\omega(\E)$ satisfying
		$\Phi\circ j_1\xi=0$\@.
	\end{enumerate}
\end{theorem}

\begin{proof}
	We shall first make some constructions that apply to both cases before
	considering each case separately.
	
	Denote $\R_1=\ker(\Phi)\subset\mathsf{J}_1\E$\@.  Let
	$k\in\mathbb{Z}_{\ge0}$\@.  Let
	$\rho_k(\Phi)\colon\mathsf{J}_{k+1}\E\to\mathsf{J}_k\F^\perp$ be the $k$th
	prolongation of $\Phi$ and denote $\R_{k+1}=\ker(\rho_k(\Phi))$\@.  Let
	$\sigma(\Phi)\colon\TsM\otimes\E\to\F^\perp$ be the symbol of $\Phi$ and
	denote by
	$\sigma_k(\Phi)\colon S^{k+1}\TsM\otimes\E\to S^k\TsM\otimes\F^\perp$ be the
	$k$th prolongation of the symbol (or, equivalently, the symbol of the $k$th
	prolongation).  Let us give an explicit formula for $\sigma_k(\Phi)$\@.
	Denote by
	\begin{equation*}
		\iota^k_{X_0}\colon S^{k+1}\TsM\to S^k\TsM
	\end{equation*}
	the contraction by $X_0$\@.  The symbol of $\Phi$ is then
	$\sigma(\Phi)=\iota^0_{X_0}\otimes P_{\F^\perp}$ and the $k$th prolongation
	of the symbol is $\sigma_k(\Phi)=\iota^k_{X_0}\otimes P_{\F^\perp}$\@.  The
	symbol of $\R_1$ is then $\mathsf{G}_1=\ker(\sigma(\Phi))=\TsM\otimes\F$ and
	the symbol of $\R_{k+1}$ is
	$\mathsf{G}_{k+1}=\ker(\sigma_k(\Phi))=S^{k+1}\TsM\otimes\F$\@.  We claim
	that $\mathsf{G}_1$ is involutive.  Indeed, the sequences
	\begin{multline*}
		0 \xrightarrowt{}  S^m\TsM\otimes\F \rd  \T^*\M\otimes S^{m-1}\TsM\otimes\F \rd \bigwedge\nolimits^{\!2} \T^*\M\otimes 
		S^{m-2}\TsM\otimes\F \rd \cdots\\  \rd  \bigwedge\nolimits^{\!{m-1}}\T^*\M\otimes \TsM\otimes\F
		\rd  \bigwedge\nolimits^{\!{m}}\T^* \M\otimes\F,
	\end{multline*}
	for $m\in\mathbb{Z}_{>0}$ are simply the $\delta$-sequences for the vector
	bundle $\F$\@, and these sequences are exact by the $\delta$-Poincar\'e
	Lemma~\autocite[Proposition~3.1.5]{JFP:78}\@.
	
	Since $\mathsf{G}_1$ is involutive and $\mathsf{G}_2$ is a subbundle, the
	only remaining ingredient to verifying the conditions
	of~\cite[Theorem~4.1]{HLG:67a} is to show that $\R_2$ projects surjectively
	onto $\R_1$.  We do this separately for the two conditions in the statement
	of the proposition.  In each case, we work with the commutative diagram
	\begin{center}
		\begin{tikzcd}
			&	& 0\arrow{d}	& 0 \arrow{d}		&		&\\
			&	&  S^2\T^*\M\otimes\E \arrow{r}{\sigma_1(\Phi)} \arrow{d}[swap]{\epsilon} & 
			\T^*\M\otimes \F^\perp \arrow{r}{\tau}\arrow{d}{\epsilon} & \mathsf{K} \arrow{r} & 0 \\
			0\arrow{r}& \R_{2}\arrow{r}\arrow{d} & \mathsf{J}_2\E \arrow{r}{\rho_1(\Phi)}\arrow{d}
			&  \mathsf{J}_1  \F^\perp\arrow{d} & & \\
			0\arrow{r} & \R_1 \arrow{r}  & \mathsf{J}_1\E \arrow{r}{\Phi} \arrow{d}&  \F^\perp\arrow{d} \\
			&	&0	&0	&	&\\
		\end{tikzcd}
	\end{center}
	with exact rows and columns.  Note that $P_{\F^\perp}\circ\nabla$ is a linear
	connection in $\F^\perp$\@,~cf.~the constrained connection
	from~\ref{par:affine-distribution}\@.  Therefore, we can use this connection
	to give a splitting
	\begin{equation*}
		\mathsf{J}_1\F^\perp\simeq\F^\perp\oplus(\TsM\otimes\F^\perp),
	\end{equation*}
	this being a simple case of~\cite[Lemma~2.15]{ADL:23a}\@.  Let
	$j_1\xi(x)\in\R_1$.  Define a map $\kappa:\R_1\to\mathsf{K}$ by
	$\kappa(j_1\xi(x)) = \tau(\beta)$\@, where $\beta\in \TsM\otimes \F^\perp$ is such
	that $\epsilon(\beta)=\rho_1(\Phi)(j_2\xi(x))$.  This is possible by the exactness of
	the third column. To show that $\R_2$ projects onto $\R_1$, we need to show
	that $\eta$ can be chosen so that $j_1\xi=j_1\eta$ and
	$j_2\eta(x)\in\R_2$. By~\autocite[Theorem~2.4.1]{JFP:78}, this is equivalent
	to $\kappa(j_1\xi(x))=0$.  Therefore, to complete the proof we show that we
	can choose $\xi$ so that $\kappa(j_1\xi(x))=0$ in each of the two cases from
	the statement of the proposition.
	
	\ref{pl:spencer1} Here we claim that the vector bundle map $\sigma_1(\Phi)$
	is surjective.  Indeed, let $\alpha\in \TsM\otimes \F$. We can consider a
	local trivialization given in terms of local sections $\{e_1,\dots,e_k\}$,
	where $k$ is the rank of $\E$, such that $\{e_1,\dots,e_m\}$ spans $\F$,
	where $m$ is the rank of $\F$, and $\{e_{m+1},\dots,e_k\}$ spans
	$\F^\perp$. Moreover, we consider a chart for $\M$ such that the chart domain
	coincides with the neighbourhood associated with the trivialization.  Since
	$X_0$ is nowhere vanishing by assumption, we can choose the chart so
	that $X_0=\partial_1$.  We can write
	$\alpha = \alpha^a_j\ \mathrm{d} x^j\otimes e_a$. We want to find
	$A\in S^2\TsM\otimes \E$ such that $P_{\F^\perp} (A(X_0)) = \alpha$. We
	construct $A$ locally. Trivially, we set $A^a_{ij} = 0$, for
	$a\in\{1,\ldots,m\}$, and
	\begin{equation*}
		A^a_{1j}  =A^a_{j1} = \alpha^a_{j},
	\end{equation*}
	for $a\in\{m+1,\ldots,k\}$ and $j\in\{1,\ldots,\dim(\M)\}$. These components
	give rise to a tensor $A$ which is mapped to $\alpha$ by $\sigma_1(\Phi)$\@.
	
	Now we note that, if $\sigma_1(\Phi)$ is surjective, then $\mathsf{K}=0$\@,
	and so $\tau$ is the zero map.  Therefore, $\kappa=0$\@.
	
	\ref{pl:spencer2} In this case, we find $\beta\in\ker(\tau)$ such that 
	$\epsilon(\beta) = \rho_1(\Phi)(j_2\xi(x))$.
	We have 
	\begin{equation*}
		\rho_1(\Phi)(j_2\xi(x))=j_1(P_{\F^\perp}\circ\nabla_{X_0}\xi)(x)
	\end{equation*}
	by definition of prolongation.  By noting that the fiber of
	$\mathsf{J}_1 \F^\perp\simeq\F^\perp\oplus(\TsM\otimes\F^\perp)$ over
	$P_{\F^\perp}(\nabla_{X_0}\xi(x))$ is just $\T^*_x\M\otimes \F^\perp_x$, we
	get
	\begin{equation*}
		\kappa(j_1\xi(x))=\tau((P_{\F^\perp}\circ\nabla)( P_{\F^\perp}\circ\nabla_{X_0}\xi)(x)),
	\end{equation*}
	since $P_{\F^\perp}\circ\nabla$ is a connection in $\F^\perp$.  Thus
	$\kappa(j_1\xi(x))=0$ if and only if
	\begin{equation*}
		(P_{\F^\perp}\circ\nabla)( P_{\F^\perp}\circ\nabla_{X_0}\xi)(x)\in\ker(\tau)=
		\image(\sigma_1(\Phi)).
	\end{equation*}
	By definition of $\sigma_1(\Phi)$\@, this condition will be met when
	$P_{\F^\perp}\circ\nabla\circ P_{\F^\perp}\circ\nabla\xi$ is symmetric for
	$\xi$ satisfying $\Phi\circ j_1\xi(x)=0$\@.
\end{proof}

\subsection{Iterative infinitesimal constructions}

Let $\pi\colon\E\to\M$ be a $\C^r$-vector bundle, let $\F\subset\E$ be a
cogeneralized subbundle, and let $\Xaff$ be an affine vector field on $\E$
over a vector field $X_0$ on $\M$\@.  We suppose that we have a linear
connection $\nabla$ in $\E$\@, and write
$\Xaff=X_0^{\mathrm{h}}+A^{\mathrm{e}}+b^{\mathrm{v}}$\@.  Motivated by
Theorem~\ref{the:reg}\@, we wish to determine ``the largest $\Xaff$-invariant
affine subbundle variety contained in $\F$\@.''  The existence of such an
object is established in~\autocite[Theorem~4.23]{ADL:20a}\@.  However, the
matter of constructing this affine subbundle variety is by no means clear.
Here we provide an infinitesimal construction based on the following
observations.
\begin{enumerate}[1.,nosep]
	\item The cogeneralized subbundle $\F$ is defined by its annihilating
	generalized subbundle $\Lambda(\F)$\@, in the sense that
	\begin{equation*}
		\F=\{w\in\E\ |\ \lambda(w)=0,\ \lambda\in\Gamma^r(\Lambda(\F))\}=
		\bigcap_{\lambda\in\Gamma^r(\Lambda(\F))}(\lambda^{\mathrm{e}})^{-1}(0).
	\end{equation*}
	Therefore, the functions $\lambda^{\mathrm{e}}$ are in the ideal sheaf of any
	variety contained in $\F$\@.
	\item If $\lambda\in\Gamma^r(\Lambda(\F))$\@, then
	$\mathscr{L}^k_{\Xaff}\lambda^{\mathrm{e}}$\@, $k\in\mathbb{Z}_{>0}$\@, will
	vanish on any flow-invariant variety contained in $\F$\@.
	\item The function $\scrLXaff \lambda^{\mathrm{e}}$\@, and more generally the
	functions $\scrLXaff^k \lambda^{\mathrm{e}}$\@, $k\in\mathbb{Z}_{>0}$\@, are
	affine functions on $\E$\@, and so their zero level sets will be affine
	subbundle contained in $\F$\@.  The intersection of these zero level sets,
	\begin{equation}\label{eq:largestaffbun}
		\bigcap_{k\in\mathbb{Z}_{\ge0}}\bigcap_{\lambda\in\Gamma^r(\Lambda(\F))}
		(\scrLXaff^k \lambda^{\mathrm{e}})^{-1}(0),
	\end{equation}
	will then be something like an affine subbundle variety invariant under
	$\Xaff$\@.  Because the intersection of affine subspaces is an affine
	subspace that is possibly empty, the projection of this affine subbundle to
	$\M$ might be a strict subset of $\M$\@, or even empty.
	\item Under suitable regularity hypotheses, the
	intersection~\eqref{eq:largestaffbun} will define the sought after largest
	$\Xaff$-invariant affine subbundle contained in $\F$\@.
\end{enumerate}
The difficulty with this procedure, and the one we address in this section,
is to come up with a suitable characterization of the iterated Lie derivates
of $\lambda^{\mathrm{e}}$ with respect to $\Xaff$\@.

We note that
\begin{equation*}
	\scrLXaff \lambda^{\mathrm{e}} =  \left(\left(\nx +
	A^*\right)\lambda\right)^{\mathrm{e}} +
	\langle\lambda\circ\pi;b\rangle^{\mathrm{h}},
\end{equation*}
using Lemma~\ref{lem:itrLieDeriv}\@.  The following result records a
recursive formulation of the iterated Lie derivatives..
\begin{proposition}\label{prop:DkBk}
	Set $L_0 = \mathrm{id}$ and $c_0=0$.  For $k\in\mathbb{Z}_{>0}$, recursively
	define $L_k$ and $c_k$ by
	\begin{equation*}
		L_k=(\nx+A^*)^k \quad \text{and}\quad
		c_k(\lambda)=\sum_{j=1}^k \scrLX^{j-1}\left\langle L_{k-j}\lambda,b\right\rangle,
	\end{equation*}
	for $U \subseteq \M$ open and $\lambda\in\calGrEs(U)$.
	Then the following statements hold:
	\begin{enumerate}[(i),nosep]
		\item \label{item:Baffkid} $c_{k+l}(\lambda)=c_l(L_k(\lambda))+\scrLX^l c_k(\lambda)$,  for $l,k\in\mathbb{Z}_{\ge 0}$\@;
		\item \label{item:scrLXaffexp}
		$\scrLXaff^k \lambda^e = (L_k(\lambda))^{\mathrm{e}}+
		(c_k(\lambda))^{\mathrm{h}}$, for $k\in\mathbb{Z}_{\geq 0}$.
	\end{enumerate}
	\begin{proof} 
		\ref{item:Baffkid} Compute
		\begin{equation*}
			\begin{split}
				c_{k+l}(\lambda) &= \sum_{j=1}^{k+l} \scrLX^{j-1}\left\langle L_{k+l-j}\lambda,b\right\rangle\\
				&= \sum_{j=1}^l \scrLX^{j-1}\left\langle L_{l-j}(L_k(\lambda)),b\right\rangle
				+ \sum_{j'=1}^k \scrLX^{l+j'-1} \left\langle L_{k-j'}(\lambda),b\right\rangle\\
				&=c_l(L_k(\lambda))+\scrLX^l c_k(\lambda).
			\end{split}
		\end{equation*}
		
		\ref{item:scrLXaffexp} We proceed by induction; for $k=1$, we compute
		\begin{equation*}
			\begin{split}
				\scrLXaff \lambda^{\mathrm{e}} &= (\nx\lambda)^{\mathrm{e}} + (A^*\lambda)^{\mathrm{e}} + (\langle\lambda,b\rangle)^{\mathrm{h}}\\
				&= (L_1(\lambda))^{\mathrm{e}} + (\langle\lambda,b\rangle)^{\mathrm{h}}.
			\end{split}
		\end{equation*}
		Suppose this is true for $k-1$ and compute
		\begin{equation*}
			\begin{split}
				\scrLXaff^k\lambda^{\mathrm{e}} &= \scrLXaff\left(
				\left(L_{k-1}(\lambda)\right)^{\mathrm{e}}+
				\left(c_{k-1}(\lambda)\right)^{\mathrm{h}}\right)\\
				&= \left(L_k(\lambda)\right)^{\mathrm{e}} + \left\langle L_{k-1}(\lambda),b\right\rangle^{\mathrm{h}} 
				+\left( \scrLX c_{k-1}(\lambda)\right)^{\mathrm{h}}\\
				&= \left(L_k(\lambda)\right)^{\mathrm{e}} +  \left\langle L_{k-1}(\lambda),b\right\rangle^{\mathrm{h}} 
				+ (c_k(\lambda) - c_1(L_{k-1}\lambda))^{\mathrm{h}}\\
				&= \left(L_k(\lambda)\right)^{\mathrm{e}} +
				\left(c_k(\lambda)\right)^{\mathrm{h}},
			\end{split}
		\end{equation*}
		where we use linearity of pullback,~\ref{item:Baffkid}\@, and the definition
		of $c_1(\lambda)$\@.
	\end{proof}
\end{proposition} 

Of course, the preceding result does not give a very concrete representation
of the iterated Lie derivatives $\scrLXaff^k\lambda^{\mathrm{e}}$\@, but it
is all one can expect.  In terms of using these iterated Lie derivatives to
find a flow-invariant affine subbundle variety, we have the following result.
\begin{proposition}\label{prop:compute-invbun}
	Let\/ $\F\subset\E$ be a\/ $\C^r$-subbundle and let\/
	$\lambda^1,\dots,\lambda^m\in\Gamma^r(\Lambda(\E))$ be generators for\/
	$\Gamma^r(\Lambda(\F))$ as a\/ $\C^r$-module.  Suppose that there exists\/
	$N\in\mathbb{Z}_{>0}$ such that
	\begin{equation*}
		\bigcap_{k\in\mathbb{Z}_{\ge0}}\bigcap_{\lambda\in\Gamma^r(\Lambda(\F))}
		(\scrLXaff^k \lambda^{\mathrm{e}})^{-1}(0)=\mathsf{A}(\Xaff,\F):=
		\bigcap_{k=0}^N\bigcap_{j=1}^m(\scrLXaff^k (\lambda^j)^{\mathrm{e}})^{-1}(0),
	\end{equation*}
	and that either\/ $r=\omega$ or that\/ $r=\infty$ and\/
	$\mathsf{A}(\Xaff,\F)$ is a submanifold.  Then\/ $\mathsf{A}(\Xaff,\F)$ is a
	flow-invariant affine subbundle variety.
	\begin{proof}
		By~\autocite[Proposition~4.3]{ADL:20a}\@, to show flow-invariance it suffices
		to show invariance.  To show invariance, it suffices to show that
		$\scrLXaff{}f$ vanishes on $\mathsf{A}(\Xaff,\F)$ for
		\begin{equation*}
			f\in\{\scrLXaff^k(\lambda^j)^{\mathrm{e}}\ |\ k\in\{0,1,\dots,N\},\
			j\in\{1,\dots,m\}\}.
		\end{equation*}
		This, however, holds by hypothesis.
		
		To show that $\mathsf{A}(\Xaff,\F)$ is an affine subbundle variety, note that
		the functions $\scrLXaff^k(\lambda^j)^{\mathrm{e}}$\@,
		$k\in\{0,1,\dots,N\}$\@, $j\in\{1,\dots,m\}$\@, are affine functions, and so
		can be identified as sections of $\E^*\oplus\mathbb{R}$\@.  Indeed, these are
		generators for a defining subbundle of $\E^*\oplus\mathbb{R}$ whose
		associated affine subbundle variety is exactly $\mathsf{A}(\Xaff,\F)$\@.
	\end{proof}
\end{proposition}

The global generators hypothesized in the statement of the result exist by an
appropriate version of the Serre\textendash{}Swan Theorem~\autocite[e.g.,][Theorem~20]{ADL:23b}\@.

\section{Example: Rolling disc on an inclined plane}\label{sec:example}

In this section, we consider a rolling disc over an inclined plane without
slip, as depicted in Figure~\ref{fig:disc}\@.
\begin{figure}[htbp]
	\centering
	\includegraphics{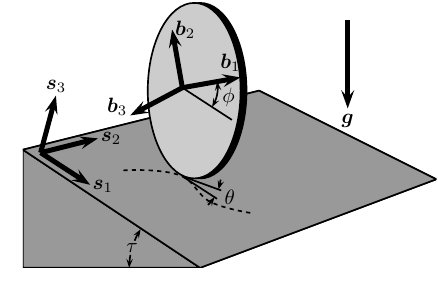}\label{fig:disc}
	\caption{Disc rolling on an inclined plane}
\end{figure}%
A classical treatment of this example is given in~\autocite{NAL:22}\@.  As we
shall see, by properly dealing with the singularity in the generalized
subbundle spanned by sections of $\D$ and their Lie brackets, we are able to
refine the analysis in~\autocite{NAL:22}\@.  \rtxt{Another example that the
	reader may find interesting to approach in a similar way is that of the
	two-wheeled carriage \autocite{MJ/WR:18,Crampin_2009}\@.}

We shall first present a few of the calculations that are needed, and then
apply the main results of the paper to draw conclusions concerning which
constrained variational trajectories are also nonholonomic trajectories.  We
use \textsc{Mathematica}\reg\ for symbolic computations.  The modelling
procedure we give is as outlined in Chapter~3 of~\autocite{FB/ADL:04}\@.

\begin{paragr}
	The spatial frame $\{\mathbf{s}_1,\mathbf{s}_2,\mathbf{s}_3\}$ is chosen so
	that $\mathbf{s}_3$ is perpendicular to the plane on which the disc
	rolls. The spin angle is denoted by $\theta$, the rolling angle is denoted by
	$\phi$, and the inclination angle for the plane is denoted by $\tau$.  The
	body frame $\{\mathbf{b}_1,\mathbf{b}_2,\mathbf{b}_3\}$ is located at the
	geometric center of the disc, with $\mathbf{b}_3$ chosen to be the axis about
	which the disc rolls, \rtxt{so that the center coordinates are $(x,y,R)$,
		where $R$ is the radius of the disc.}
\end{paragr}

\begin{paragr}
	From the problem setup, the configuration manifold is
	\begin{equation*}
		\Q=\left\{
		\begin{bmatrix}
			\cos(\phi)\cos(\theta)& -\sin(\phi)\cos(\theta)& \sin(\theta)\\
			\cos(\phi)\sin(\theta)& -\sin(\phi)\sin(\theta)& -\cos(\theta)\\
			\sin(\phi)& \cos(\phi)& 0
		\end{bmatrix},(x,y,R)
		\right\}\subseteq \mathrm{SO}(3)\times\bbR^3.
	\end{equation*}
	It is easy to see that $\Q\simeq \bbT^2\times \bbR^2$. We carry out our
	calculations in a single chart with coordinates
	$(\theta,\phi,x,y)\in (-\pi,\pi)^2\times \bbR^2$. We abuse the notation and
	identify the chart domain with $\Q$.  We will denote a typical point in $\Q$ as
	$q$\@.  It will be convenient at times to use the coordinate functions as
	functions, and write $x(q)$ or $\phi(q)$\@, for example.
\end{paragr}

\begin{paragr}
	We assume that the inertia tensor of the disc is represented by a diagonal
	matrix $\bbI$ given by $\bbI_{11}=\bbI_{22}= J_s$ and $\bbI_{33}=J_r$, where
	$J_s$ denotes the spinning principal inertia associated with $\mathbf{b}_1$
	and $\mathbf{b}_2$, and $J_r$ is the principal inertia related to rolling.
	We assume that $\mathbf{b}_3$ is an axis of symmetry of the body.  We denote
	the mass of the disc by $m$\@.  One then readily ascertains that
	\begin{equation*}
		\bbG = m dx\otimes dx + m dy\otimes dy +
		J_sd\theta\otimes d\theta + J_rd\phi\otimes d\phi.
	\end{equation*}
	Clearly all Christoffel symbols for the associated Levi-Civita connection
	vanish.
\end{paragr}

\begin{paragr}
	The potential energy is given by $V(\theta,\phi,x,y) = mg(R-x\sin(\tau))$\@,
	where $g$ is the acceleration due to gravity.  Hence we compute
	\begin{equation*}
		\mathrm{grad}\ V = \bbG^\#\circ dV = -\frac{g\sin(\tau)}{m}\partial_x.
	\end{equation*}
\end{paragr}

\begin{paragr}
	By assumption, the disc rolls with no-slip on the inclined plan.  Hence, for
	any non-holonomic trajectory $t\mapsto (\theta(t),\phi(t),x(t),y(t))$, we
	must have
	\begin{equation*}
		\begin{split}
			\dot{x}(t) &= R\dot{\phi}(t)\cos(\theta(t)),\\
			\dot{y}(t) &= R\dot{\phi}(t)\sin(\theta(t)).
		\end{split}
	\end{equation*}
	We define a vector bundle map $F\colon \TQ\to \Q\times \bbR^2$ by
	\begin{equation*}
		F(\theta,\phi,x,y,v_\theta,v_\phi,v_x,v_y)= (\theta,\phi,x,y, v_x -
		Rv_\phi\cos(\theta), v_y -Rv_\phi\sin(\theta)).
	\end{equation*}
	We set $\D=\ker(F)$. It is clear that $\D$ has constant rank and $\D$ is the
	constraint distribution.
\end{paragr}

\begin{paragr}
	Let us determine a suitable basis for $\TQ$ adapted to the distribution $\D$
	and its $\bbG$-orthogonal complement $\D^\perp$\@.  By construction, the
	vector fields
	\begin{equation*}
		\left\{\partial_\theta,\partial_\phi  + R\cos(\theta) \partial_x +
		R\sin(\theta)\partial_y\right\}
	\end{equation*}
	span $\D$. In fact, one can see that the given vector fields are
	$\bbG$-orthogonal.  One can thus normalize to obtain $\bbG$-orthonormal
	vector fields $X_1,X_2\in\Gamma^r(\D)$ spanning $\D$. That is,
	\begin{equation*}
		X_1(\theta,\phi,x,y) = \frac{1}{\sqrt{J_s}} \partial_\theta 
	\end{equation*}
	and 
	\begin{equation*}
		X_2(\theta,\phi,x,y) = \frac{1}{\sqrt{J_r + m R^2  }}\left( 
		\partial_\phi  + R\cos(\theta) \partial_x + R\sin(\theta)\partial_y\right).
	\end{equation*}
	A convenient orthonormal basis for $\D^\perp$ is verified to be given by
	$\{X_3,X_4\}$ with
	\begin{equation*}
		X_3(\theta,\phi,x,y) =
		\frac{\sin(\theta)}{\sqrt{m}}\partial_x-
		\frac{\cos(\theta)}{\sqrt{m}}\partial_y
	\end{equation*}
	and
	\begin{equation*}
		X_4(\theta,\phi,x,y) =
		\frac{\sqrt{J_r}\cos(\theta)}{\sqrt{m}\sqrt{J_r+mR^2}}\partial_x+
		\frac{\sqrt{J_r}\sin(\theta)}{\sqrt{m}\sqrt{J_r+mR^2}}\partial_x-
		\frac{\sqrt{m}}{\sqrt{J_r}\sqrt{J_r+mR^2}}\partial_\phi.
	\end{equation*}
\end{paragr}

\begin{paragr}
	The bases $\{X_1,X_2\}$ and $\{X_3,X_4\}$ for $\D$ and $\D^\perp$\@,
	respectively, allow us to introduce coordinates for $\D$ and $\D^\perp$\@.
	We write a typical point in $\D$ as $v_sX_1+v_rX_1$\@, where $v_s$ is the
	``spin velocity'' and $v_r$ is the ``roll velocity.''  Thus we use
	coordinates $(\theta,\phi,x,y,v_r,v_s)$ for $\D$\@.
	
	We will denote a typical point in $\D$ by $v$ and, as we indicated above for
	coordinates for $\Q$\@, we will think of coordinates for $\D$ as functions on
	$\D$\@.  Thus we will sometimes write $y(v)$\@, or $\theta(v)$\@, or
	$v_r(v)$\@, for example.
\end{paragr}

\begin{paragr}
	Note that the basis vector fields for $\D$ and $\D^\perp$ give basis vector
	fields for the pull-back bundles $\pi_\D^*\D$ and $\pi_\D^*\D^\perp$\@, and
	we denote these bases by
	\begin{equation*}
		\{\pi_\D^*X_1,\pi_\D^*X_2\},\quad\{\pi_\D^*X_3,\pi_\D^*X_4\},
	\end{equation*}
	respectively.  There are also dual bases that we denote by
	\begin{equation*}
		\{\pi_\D^*X^*_1,\pi_\D^*X^*_2\},\quad\{\pi_\D^*X^*_3,\pi_\D^*X^*_4\}.
	\end{equation*}
	
	Using the basis for $\pi_\D^*\D^\perp$\@, we write a typical point in
	$\pi_\D^*\D^\perp$ as $p_1\pi_\D^*X_3+p_2\pi_\D^*X_4$\@.  This gives
	coordinates $(\theta,\phi,x,y,v_s,v_r,p_1,p_2)$ for $\pi_\D^*\D^\perp$\@.  A
	typical point in $\pi_\D^*\D^\perp$ we denote by $p$\@, and we think of the
	coordinates as functions on $\pi_\D^*\D^\perp$\@.  Thus we may write
	$x(p)$\@, or $\phi(p)$\@, or $v_s(p)$\@, or $p_1(p)$\@.
\end{paragr}

\begin{paragr}
	We calculate
	\begin{align*}
		\nabg_{X_1}X_2 =&\; -\frac{\sqrt{m}R}{\sqrt{J_s}\sqrt{J_r+mR^2}}X_3,\\
		\nabg_{X_1}X_3 =&\; \frac{\sqrt{m}R}{\sqrt{J_s}\sqrt{J_r+mR^2}}X_2+
		\frac{\sqrt{J_r}}{\sqrt{J_s}\sqrt{J_r+mR^2}}X_4,\\
		\nabg_{X_1}X_4 =&\; \frac{\sqrt{J_r}}{\sqrt{J_s}\sqrt{J_r+mR^2}}X_3,
	\end{align*}
	and all other covariant derivatives of the basis vector fields are zero.
	From these computations, we make a few immediate conclusions.
	\begin{enumerate}[1.,nosep]
		\item The Christoffel symbols of the linear connection $\nabd=P_\D\circ\nabg$
		in $\D$ with respect to the frames $\{X_1,X_2,X_3,X_4\}$ for $\TQ$ and
		$\{X_1,X_2\}$ for $\D$ are zero.
		\item The Christoffel symbols of the linear connection
		$\nabdp=P_{\D^\perp}\circ\nabg$ in $\D^\perp$ with respect to the frames
		$\{X_1,X_2,X_3,X_4\}$ for $\TQ$ and $\{X_3,X_4\}$ $\D^\perp$ are determined
		by
		\begin{equation*}
			\nabdp_{X_1}X_3=\frac{\sqrt{J_r}}{\sqrt{J_s}\sqrt{J_r+mR^2}}X_4,\quad
			\nabdp_{X_1}X_4=\frac{\sqrt{J_r}}{\sqrt{J_s}\sqrt{J_r+mR^2}}X_3.
		\end{equation*}
		\item The Frobenius curvature $F_\D$ is determined by
		\begin{equation*}
			F_\D(X_1,X_2)=-F_D(X_2,X_1)=-\frac{\sqrt{m}R}{\sqrt{J_s}\sqrt{J_r+mR^2}}X_3.
		\end{equation*}
		\item The geodesic curvature $G_\D$ is determined by
		\begin{equation*}
			G_\D(X_1,X_2)=G_D(X_2,X_1)=-\frac{\sqrt{m}R}{\sqrt{J_s}\sqrt{J_r+mR^2}}X_3.
		\end{equation*}
		\item The Frobenius and geodesic curvatures $F_{\D^\perp}$ and $G_{\D^\perp}$
		both vanish; thus $\D^\perp$ is integrable and geodesically invariant for
		$\nabg$\@.
	\end{enumerate}
\end{paragr}

\begin{paragr}
	Let us next calculate $\hat{F}_\D^*$\@, which is a vector bundle map from
	$\pi_\D^*\D^\perp$ to $\pi_\D^*\D$\@.  We represent this bundle map as a
	$2\times2$ matrix representing this map in the bases
	$\{\pi_\D^*X_3,\pi_\D^*X_4\}$ and
	$\{\pi_\D^*X_1,\pi_\D^*X_2\}$\@.  Using the computation of
	$F_\D$ from above, we ascertain that
	\begin{equation*}
		[\hat{F}_\D^*]=\begin{bmatrix}
			-\frac{\sqrt{m}R}{\sqrt{J_s}\sqrt{J_r+mR^2}}v_r&0\\
			\frac{\sqrt{m}R}{\sqrt{J_s}\sqrt{J_r+mR^2}}v_s&0\end{bmatrix}.
	\end{equation*}
	We observe that
	\begin{equation*}
		\ker(\hat{F}_\D^*)_v=
		\begin{cases}\Span\{\pi_\D^*X_4\},&v_s^2+v_r^2\not=0,\\
			(\pi_\D^*\D^\perp)_v,&v_s=v_r=0,\end{cases}
	\end{equation*}
	where $v\in\D$\@.
\end{paragr}

\begin{paragr}
	Let us next calculate $A_\D$ and $b_\D$\@.  Since $G_{\D^\perp}$ and
	$F_{\D^\perp}$ are zero, $A_\D$ is also zero.  We note that $b_\D$ is a
	section of the pullback bundle $\pi_\D^*\D^\perp$\@.  We represent $b_\D$ in
	terms of its components in the basis provided by $\pi_\D^*X_3$ and
	$\pi_\D^*X_4$\@.  A calculation gives
	\begin{equation*}
		b_\D=-\left(\sqrt{m}g\sin(\tau)\sin(\theta)+
		\frac{2R}{\sqrt{J_s}\sqrt{J_r+mR^2}}v_sv_r\right)\pi_\D^*X_3-
		\frac{g\sqrt{m}\sqrt{J_r}}{\sqrt{J_r+mR^2}}\sin(\tau)\cos(\theta)\pi_\D^*X_4.
	\end{equation*}
\end{paragr}

\begin{paragr}
	Let us determine the vector field $\Xnh$ using the ``Poincar\'e
	representation'' explained in \autocite[\S4.6.4]{FB/ADL:04}\@.  We can use
	the computations above to ascertain that
	\begin{multline*}
		\Xnh(\theta,\phi,x,y,v_s,v_r)\\
		=\frac{v_s}{\sqrt{J_s}}\partial_\theta+
		\frac{v_r}{\sqrt{J_r+mR^2}}\partial_\phi+
		\frac{R\cos(\theta)v_r}{\sqrt{J_r+mR^2}}\partial_x+
		\frac{R\sin(\theta)v_r}{\sqrt{J_r+mR^2}}\partial_y+
		\frac{mgR\sin(\tau)\cos(\theta)}{\sqrt{J_r+mR^2}}\partial_{v_r}.
	\end{multline*}
	Thus the associated differential equations are
	\begin{align*}
		\dot{x}(t)=&\;\frac{R\cos(\theta(t))v_r(t)}{\sqrt{J_r+mR^2}},\\
		\dot{y}(t)=&\;\frac{R\sin(\theta(t))v_r(t)}{\sqrt{J_r+mR^2}},\\
		\dot{\theta}(t)=&\;\frac{v_s(t)}{\sqrt{J_s}},\\
		\dot{\phi}(t)=&\;\frac{v_r(t)}{\sqrt{J_r+mR^2}},\\
		\dot{v}_s(t)=&\;0,\\
		\dot{v}_r(t)=&\;\frac{mgR\sin(\tau)\cos(\theta(t))}{\sqrt{J_r+mR^2}}.
	\end{align*}
\end{paragr}

\begin{paragr}
	We can similarly ascertain the equations governing the regular constrained
	variational trajectories.  Here we again use the Poincar\'e representation
	for the vector fields.  We compute the vector field on $\pi_\D^*\D^\perp$
	representing the dynamics~\eqref{eq:dreg} as
	\begin{multline*}
		X_\D^{\mathrm{rcv}}=
		\frac{R\cos(\theta)v_r}{\sqrt{J_r+mR^2}}\partial_x+
		\frac{R\sin(\theta)v_r}{\sqrt{J_r+mR^2}}\partial_y+
		\frac{v_s}{\sqrt{J_s}}\partial_\theta+
		\frac{v_r}{\sqrt{J_r+mR^2}}\partial_\phi\\
		-\frac{\sqrt{m}v_rp_1}{\sqrt{J_s}\sqrt{J_r+mR^2}}\partial_{v_s}+
		\frac{mgR\sqrt{J_s}\cos(\theta)\sin(\tau)+
			\sqrt{m}v_sp_1}{\sqrt{J_s}\sqrt{J_r+mR^2}}\partial_{v_r}\\
		+\left(\frac{\sqrt{J_r}v_sp_2}{\sqrt{J_s}\sqrt{J_r+mR^2}}
		-\frac{2\sqrt{m}Rv_sv_r}{\sqrt{J_s}\sqrt{J_r+mR^2}}
		-\sqrt{m}g\sin(\tau)\sin(\theta)\right)\partial_{p_1}\\
		-\left(\frac{\sqrt{J_r}v_sp_1}{\sqrt{J_s}\sqrt{J_r+mR^2}}+
		\frac{\sqrt{m}\sqrt{J_r}g\sin(\tau)\cos(\theta)}{\sqrt{J_r+mR^2}}\right)
		\partial_{p_2}.
	\end{multline*}
	The differential equations are then computed to be
	\begin{align*}
		\dot{x}(t)=&\;\frac{R\cos(\theta(t))v_r(t)}{\sqrt{J_r+mR^2}},\\
		\dot{y}(t)=&\;\frac{R\sin(\theta(t))v_r(t)}{\sqrt{J_r+mR^2}},\\
		\dot{\theta}(t)=&\;\frac{v_s(t)}{\sqrt{J_s}},\\
		\dot{\phi}(t)=&\;\frac{v_r(t)}{\sqrt{J_r+mR^2}},\\
		\dot{v}_s(t)=&\;-\frac{\sqrt{m}Rv_r(t)p_1(t)}{\sqrt{J_s}\sqrt{J_r+mR^2}},\\
		\dot{v}_r(t)=&\;\frac{mgR\sqrt{J_s}\cos(\theta(t))\sin(\tau)+
			\sqrt{m}Rv_s(t)p_1(t)}{\sqrt{J_s}\sqrt{J_r+mR^2}},\\
		\dot{p}_1(t)=&\;\frac{\sqrt{J_r}v_s(t)p_2(t)}{\sqrt{J_s}\sqrt{J_r+mR^2}}
		-\frac{2\sqrt{m}Rv_s(t)v_r(t)}{\sqrt{J_s}\sqrt{J_r+mR^2}}
		-\sqrt{m}g\sin(\tau)\sin(\theta(t)),\\
		\dot{p}_2(t)=&\;-\frac{\sqrt{J_r}v_s(t)p_1(t)}{\sqrt{J_s}\sqrt{J_r+mR^2}}
		-\frac{\sqrt{m}\sqrt{J_r}g\sin(\tau)\cos(\theta(t))}{\sqrt{J_r+mR^2}}.
	\end{align*}
\end{paragr}

\begin{paragr}
	One can also ascertain that the singular constrained trajectories as
	prescribed by~\eqref{eq:dsing} are determined by curves
	$t\mapsto\gamma(t)\in\Q$ and nonzero sections $t\mapsto\lambda(t)$ of
	$\D^\perp$ over $\gamma$ satisfying the algebraic equation
	\begin{align*}
		\frac{\sqrt{m}R}{\sqrt{J_s}\sqrt{J_r+mR^2}}v_r(t)p_1(t)=&\;0,\\
		\frac{\sqrt{m}R}{\sqrt{J_s}\sqrt{J_r+mR^2}}v_s(t)p_2(t)=&\;0
	\end{align*}
	(of course, one can eliminate the coefficient
	$\frac{\sqrt{m}R}{\sqrt{J_s}\sqrt{J_r+mR^2}}$\@, but we elect to keep it
	because it reminds us where these conditions come from), and the differential
	equation
	\begin{align*}
		\dot{p}_1(t)=&\;\frac{\sqrt{J_r}v_s(t)p_2(t)}{\sqrt{J_s}\sqrt{J_r+mR^2}},\\
		\dot{p}_2(t)=&\;-\frac{\sqrt{J_r}v_s(t)p_1(t)}{\sqrt{J_s}\sqrt{J_r+mR^2}}.
	\end{align*}
\end{paragr}

\begin{paragr}
	We now use Proposition~\ref{prop:compute-invbun} to find the largest
	$X_\D^{\mathrm{reg}}$-invariant affine subbundle variety of
	$\pi_\D^*\D^\perp$ contained in $\ker(\hat{F}_\D^*)$\@, as required by
	Theorem~\ref{the:reg} to determine the regular constrained variational
	trajectories that are also nonholonomic trajectories.  We have
	\begin{multline*}
		X_\D^{\mathrm{reg}}=
		\frac{R\cos(\theta)v_r}{\sqrt{J_r+mR^2}}\partial_x+
		\frac{R\sin(\theta)v_r}{\sqrt{J_r+mR^2}}\partial_y+
		\frac{v_s}{\sqrt{J_s}}\partial_\theta+
		\frac{v_r}{\sqrt{J_r+mR^2}}\partial_\phi
		+\frac{mgR\sqrt{J_s}\sin(\tau)\cos(\theta)}
		{\sqrt{J_s}\sqrt{J_r+mR^2}}\partial_{v_r}\\
		+\left(\frac{\sqrt{J_r}v_sp_2}{\sqrt{J_s}\sqrt{J_r+mR^2}}
		-\frac{2\sqrt{m}Rv_sv_r}{\sqrt{J_s}\sqrt{J_r+mR^2}}
		-\sqrt{m}g\sin(\tau)\sin(\theta)\right)\partial_{p_1}\\
		-\left(\frac{\sqrt{J_r}v_sp_1}{\sqrt{J_s}\sqrt{J_r+mR^2}}+
		\frac{\sqrt{m}\sqrt{J_r}g\sin(\tau)\cos(\theta)}{\sqrt{J_r+mR^2}}\right)
		\partial_{p_2},
	\end{multline*}
	and the associated differential equations are
	\begin{equation}\label{eq:diskXreg}
		\begin{aligned}
			\dot{x}(t)=&\;\frac{R\cos(\theta(t))v_r(t)}{\sqrt{J_r+mR^2}},\\
			\dot{y}(t)=&\;\frac{R\sin(\theta(t))v_r(t)}{\sqrt{J_r+mR^2}},\\
			\dot{\theta}(t)=&\;\frac{v_s(t)}{\sqrt{J_s}},\\
			\dot{\phi}(t)=&\;\frac{v_r(t)}{\sqrt{J_r+mR^2}},\\
			\dot{v}_s(t)=&\;0,\\
			\dot{v}_r(t)=&\;\frac{mgR\sqrt{J_s}\sin(\tau)\cos(\theta(t))}
			{\sqrt{J_s}\sqrt{J_r+mR^2}},\\
			\dot{p}_1(t)=&\;\frac{\sqrt{J_r}v_s(t)p_2(t)}{\sqrt{J_s}\sqrt{J_r+mR^2}}
			-\frac{2\sqrt{m}Rv_s(t)v_r(t)}{\sqrt{J_s}\sqrt{J_r+mR^2}}
			-\sqrt{m}g\sin(\tau)\sin(\theta(t)),\\
			\dot{p}_2(t)=&\;-\frac{\sqrt{J_r}v_s(t)p_1(t)}{\sqrt{J_s}\sqrt{J_r+mR^2}}
			-\frac{\sqrt{m}\sqrt{J_r}g\sin(\tau)\cos(\theta(t))}{\sqrt{J_r+mR^2}}.
		\end{aligned}
	\end{equation}
	We note that
	\begin{equation*}
		\ker(\hat{F}_\D^*)=(\lambda^{\mathrm{e}})^{-1}(0),
	\end{equation*}
	where $\lambda\in(\pi_\D^*\D^\perp)^*$ is given by
	\begin{equation*}
		\lambda=(v_s^2+v_r^2)\pi_\D^*X_3^*.
	\end{equation*}
	One computes
	\begin{multline*}
		\scrL_{X_\D^{\mathrm{reg}}}\lambda^{\mathrm{e}}=
		\frac{2mgR\sin(\tau)\cos(\theta)}{\sqrt{J_r+mR^2}}v_rp_1\\
		+(v_s^2+v_r^2)\left(\frac{\sqrt{J_r}}{\sqrt{J_s}\sqrt{J_r+mR^2}}v_sp_2-
		\sqrt{m}g\sin(\tau)\sin(\theta)-
		\frac{2\sqrt{m}R}{\sqrt{J_s}\sqrt{J_r+mR^2}}v_sv_r\right).
	\end{multline*}
	One can also compute $\scrL^k_{X_\D^{\mathrm{reg}}}\lambda^{\mathrm{e}}$ for
	$k\ge2$\@, but the symbolic expressions are too unwieldy to record.
	
	Our strategy for finding invariant affine subbundle varieties is the
	following.  For $k\in\mathbb{Z}_{\ge0}$\@, denote
	\begin{equation*}
		\A_k(X_\D^{\mathrm{reg}},\lambda^{\mathrm{e}})=
		\bigcap_{j=0}^k(\scrL^k_{X_\D^{\mathrm{reg}}}\lambda^{\mathrm{e}})^{-1}(0).
	\end{equation*}
	We shall sequentially consider conditions on $p\in\pi_\D^*\D^\perp$ that
	ensure that $p\in\A_k(X_\D^{\mathrm{reg}},\lambda^{\mathrm{e}})$\@,
	$k\in\mathbb{Z}_{\ge0}$\@.  We do this explicitly and in detail for
	$k\in\{0,1\}$\@, and then use our conclusions, along with a recording of
	conclusions deduced from \textsc{Mathematica}\reg\ and the differential
	equations~\eqref{eq:diskXreg}\@, to finalize the conditions on $p$\@.
	
	\begin{enumerate}[R1.,nosep]
		\item We first take $\sin(\tau)\not=0$\@.  If
		$p\in\A_0(X_\D^{\mathrm{reg}},\lambda^{\mathrm{e}})$\@, then either
		(a)~$p_1(p)=0$ or (b)~$v_s(p)=v_r(p)=0$\@.  We consider these possibilities
		in turn.
		
		\begin{enumerate}[(a),nosep]
			\item If $p_1(p)=0$\@, then
			$p\in\A_1(X_\D^{\mathrm{reg}},\lambda^{\mathrm{e}})$ if and only if either
			(i)~$v_s(p)=v_r(p)=0$ or~(ii)
			\begin{equation*}
				\frac{\sqrt{J_r}}{\sqrt{J_s}\sqrt{J_r+mR^2}}v_s(p)p_2(p)-
				\sqrt{m}g\sin(\tau)\sin(\theta(p))-
				\frac{2\sqrt{m}R}{\sqrt{J_s}\sqrt{J_r+mR^2}}v_s(p)v_r(p)=0.
			\end{equation*}
			Let us denote by $\zeta(p)$ the expression on the left in the preceding
			equation.  We consider the preceding possibilities in turn.
			
			\begin{enumerate}[(i),nosep]
				\item If
				\begin{equation*}
					p\in\A'_1(X_\D^{\mathrm{reg}},\lambda^{\mathrm{e}}):=
					\{p'\in\pi_{\D}^*\D^\perp\ |\ p_1(p')=0,\ v_s(p')=0,\ v_r(p')=0\}, 
				\end{equation*}
				then we calculate
				\begin{align*}
					\scrL^2_{X_\D^{\mathrm{reg}}}\lambda^{\mathrm{e}}(p)=&\;0,\\
					\scrL^3_{X_\D^{\mathrm{reg}}}\lambda^{\mathrm{e}}(p)=&\;
					-\frac{6m^{5/2}g^3R^2\sin(\tau)^3}{J_r+mR^2}\cos(\theta(p))^2\sin(\theta(p)).
				\end{align*}
				This last expression vanishes exactly when either $\sin(\theta(p))=0$ or
				$\cos(\theta(p))=0$\@.  If we take
				\begin{multline*}
					\A=\{p\in\pi_{\D}^*\D^\perp\ |\ p_1(p)=0,\ v_s(p)=0,\ v_r(p)=0,\
					\sin(\theta(p))=0\}\\\cup
					\{p\in\pi_{\D}^*\D^\perp\ |\ p_1(p)=0,\ v_s(p)=0,\ v_r(p)=0,\
					\cos(\theta(p))=0\},
				\end{multline*}
				then the differential equations~\eqref{eq:diskXreg} are easily examined to
				see that $\A$ is an $X_\D^{\mathrm{reg}}$-invariant affine subbundle
				variety.  We observe that $p_2$ is constant along integral curves of
				$X_\D^{\mathrm{reg}}$ with initial conditions in $\A$\@.
				
				(We remark that the submanifold defined by the condition $v_s(p)=0$ is
				$X_\D^{\mathrm{reg}}$-invariant.  That is, integral curves of
				$X_\D^{\mathrm{reg}}$ with an initial condition with $v_s(0)=0$ will satisfy
				$v_s(t)=0$ for all $t$\@.  In this case, one also sees that $\theta$ is
				constant.  This is a starting point for much of the analysis of the
				differential equations~\eqref{eq:diskXreg} that we reference below.)
				
				\item For the condition that $\zeta(p)=0$ we consider two cases, namely
				(I)~$v_s(p)=0$ and (II)~$v_s(p)\not=0$\@.  Let us consider these cases in
				turn.
				\begin{enumerate}[(I),nosep]
					\item If
					\begin{equation*}
						p\in\A'_1(X_\D^{\mathrm{reg}},\lambda^{\mathrm{e}}):=
						\{p'\in\pi_{\D}^*\D^\perp\ |\ p_1(p')=0,\ v_s(p')=0\},
					\end{equation*}
					then we calculate
					\begin{equation*}
						\scrL_{X_\D^{\mathrm{reg}}}\lambda^{\mathrm{e}}(p)=
						-\sqrt{m}g\sin(\tau)\sin(\theta(p))v_r(p)^2.
					\end{equation*}
					Excluding the condition $v_r(p)=0$ which we have already considered above,
					the vanishing of the expression on the right requires that
					$\sin(\theta(p))=0$\@.  If we take
					\begin{equation*}
						\A=\{p\in\pi_{\D}^*\D^\perp\ |\ p_1(p)=0,\ v_s(p)=0,\ \sin(\theta(p))=0\},
					\end{equation*}
					then the differential equations~\eqref{eq:diskXreg} are easily examined to
					see that $\A$ is an $X_\D^{\mathrm{reg}}$-invariant affine subbundle
					variety.  The evolution of $p_2$ along integral curves is as determined by
					the $p_2$-component of the differential equations~\eqref{eq:diskXreg}\@,
					noting that $p_1=0$ in this case.
					
					\item In this case, the condition $\zeta(p)=0$ uniquely specifies $p_2(p)$\@,
					and we denote this value of $p_2$ by $p_2^*(v)$\@, noting that this value
					depends only on $v=\pi_\D^*\pi_{\D^\perp}(p)$\@.  If
					\begin{equation*}
						p\in\A'_1(X_\D^{\mathrm{reg}},\lambda^{\mathrm{e}}):=
						\{p'\in\pi_{\D}^*\D^\perp\ |\ p_1(p')=0,\ p_2(p')=p_2^*(v')\},
					\end{equation*}
					then we calculate
					\begin{equation*}
						\scrL^2_{X_\D^{\mathrm{reg}}}\lambda^{\mathrm{e}}(p)=
						\alpha_1(v)\cos(\theta(p)),
					\end{equation*}
					where $\alpha_1$ is a complicated nowhere zero $\theta$-independent function
					of the points $v\in\D$ satisfying $v_s(v)\not=0$\@.  Now we let
					\begin{multline*}
						p\in\A'_1(X_\D^{\mathrm{reg}},\lambda^{\mathrm{e}})\\:=
						\{p'\in\pi_{\D}^*\D^\perp\ |\ p_1(p')=0,\ p_2(p')=p_2^*(v'),\
						\cos(\theta(p'))=0\},
					\end{multline*}
					then we calculate
					\begin{equation*}
						\scrL^3_{X_\D^{\mathrm{reg}}}\lambda^{\mathrm{e}}(p)=
						\alpha_2(v)\sin(\theta(p)),
					\end{equation*}
					where $\alpha_2$ is a complicated nowhere zero $\theta$-independent function
					of the points $v\in\D$ satisfying $v_s(v)\not=0$\@.  Since $\sin$ and $\cos$
					have no common zeros, we see that the subset
					\begin{equation*}
						\{p\in\pi_\D^*\D^\perp\ |\ p_1(p)=0,\ v_s(p)\not=0\}
					\end{equation*}
					contains no $X_\D^{\mathrm{reg}}$-invariant affine subbundle varieties.
				\end{enumerate}
			\end{enumerate}
			
			\item If
			\begin{equation*}
				p\in\A'_1(X_\D^{\mathrm{reg}},\lambda^{\mathrm{e}}):=
				\{p'\in\pi_\D^*\D^\perp\ |\ v_s(p)=0,\ v_r(p)=0\},
			\end{equation*}
			then
			\begin{equation*}
				\scrL^2_{X_\D^{\mathrm{reg}}}\lambda^{\mathrm{e}}(p)=
				\frac{2m^2g^2R^2\sin(\tau)^2}{J_r+mR^2}\cos(\theta(p))^2p_1(p).
			\end{equation*}
			Setting aside the condition that $p_1(p)=0$ that has already been considered,
			the vanishing of the expression on the right requires that
			$\cos(\theta(p))=0$\@.  Now, if we let
			\begin{equation*}
				\A=\{p\in\pi_\D^*\D^\perp\ |\ v_s(p)=0,\ v_r(p)=0,\ \cos(\theta(p))=0\},
			\end{equation*}
			then a consideration of the differential equations~\eqref{eq:diskXreg} shows
			that $\A$ is an $X_\D^{\mathrm{reg}}$-invariant affine subbundle variety.
			The evolution of $p_1$ and $p_2$ along integral curves is as determined by
			the $p_1$- and $p_2$-components of the differential
			equations~\eqref{eq:diskXreg}\@.
		\end{enumerate}
		
		\item Now we consider the case of $\sin(\tau)=0$\@.  As above, for
		$p\in\A_0(X_\D^{\mathrm{reg}},\lambda^{\mathrm{e}})$\@, we have the two cases
		(a)~$p_1(p)=0$ or (b)~$v_s(p)=v_r(p)=0$\@, which we consider in turn.
		
		\begin{enumerate}[(a),nosep]
			\item If $p_1(p)=0$\@, then
			$p\in\A_1(X_\D^{\mathrm{reg}},\lambda^{\mathrm{e}})$ if and only
			if either (i)~$v_s(p)=v_r(p)=0$ or~(ii)
			\begin{equation*}
				\frac{\sqrt{J_r}}{\sqrt{J_s}\sqrt{J_r+mR^2}}v_s(p)p_2(p)-
				\frac{2\sqrt{m}R}{\sqrt{J_s}\sqrt{J_r+mR^2}}v_s(p)v_r(p)=0.
			\end{equation*}
			Let us again abbreviate the expression on the left by $\zeta(p)$\@.  We
			consider the two preceding cases.
			\begin{enumerate}[(i),nosep]
				\item If we take
				\begin{equation*}
					\A=\{p\in\pi_\D^*\D^\perp\ |\ p_1(p)=0,\ v_s(p)=0,\ v_r(p)=0\},
				\end{equation*}
				then we can immediately deduce from the differential
				equations~\eqref{eq:diskXreg} that $\A$ is an $X_\D^{\mathrm{reg}}$-invariant
				affine subbundle variety.  We observe that $p_2$ is constant along integral
				curves of $X_\D^{\mathrm{reg}}$ with initial conditions in $\A$\@.
				
				\item For the condition that $\zeta(p)=0$ we consider two cases, namely
				(I)~$v_s(p)=0$ and (II)~$v_s(p)\not=0$\@.  Let us consider these cases in
				turn.
				
				\begin{enumerate}[(I),nosep]
					\item In this case, we can immediately verify from the differential
					equations~\eqref{eq:diskXreg} that
					\begin{equation*}
						\A=\{p\in\pi_\D^*\D^\perp\ |\ p_1(p)=0,\ v_s(p)=0\}
					\end{equation*}
					is an $X_\D^{\mathrm{reg}}$-invariant affine subbundle variety.  We observe
					that $p_2$ is constant along integral curves of $X_\D^{\mathrm{reg}}$ with
					initial conditions in $\A$\@.
					
					\item In this case, the condition $\zeta(p)$ uniquely specifies $p_2(p)$\@,
					say $p_2(p)=p_2^*(v)$\@.  One can see that, if we take
					\begin{equation*}
						\A=\{p\in\pi_\D^*\D^\perp\ |\ p_1(p)=0,\ p_2(p)=p_2^*(v)\},
					\end{equation*}
					then both $\dot{p}_1$ and $\dot{p}_2$ are zero on $\A$\@.  Thus $\A$ is an
					$X_\D^{\mathrm{reg}}$-invariant affine subbundle variety.
				\end{enumerate}
			\end{enumerate}
			
			\item In this case, we take
			\begin{equation*}
				\A=\{p\in\pi_\D^*\D^\perp\ |\ v_s(p)=0,\ v_r(p)=0\}
			\end{equation*}
			and note directly from the differential equations~\eqref{eq:diskXreg} that
			$\A$ is an $X_\D^{\mathrm{reg}}$-invariant affine subbundle variety.  We
			observe that both $p_1$ and $p_2$ are constant along integral curves of
			$X_\D^{\mathrm{reg}}$ with initial conditions in $\A$\@.
		\end{enumerate}
	\end{enumerate}
\end{paragr}

\begin{paragr}
	We can also find the largest $X_\D^{\mathrm{sing}}$-invariant subbundle of
	$\pi_\D^*\D^\perp$ contained in $\ker(\hat{F}_\D^*)$\@, as required by
	Theorem~\ref{the:sing} to determine the singular constrained variational
	trajectories that are also nonholonomic trajectories.  We have
	\begin{multline*}
		X_\D^{\mathrm{sing}}=
		\frac{R\cos(\theta)v_r}{\sqrt{J_r+mR^2}}\partial_x+
		\frac{R\sin(\theta)v_r}{\sqrt{J_r+mR^2}}\partial_y+
		\frac{v_s}{\sqrt{J_s}}\partial_\theta+
		\frac{v_r}{\sqrt{J_r+mR^2}}\partial_\phi
		+\frac{mgR\sqrt{J_s}\sin(\tau)\cos(\theta)}
		{\sqrt{J_s}\sqrt{J_r+mR^2}}\partial_{v_r}\\
		+\frac{\sqrt{J_r}v_sp_2}{\sqrt{J_s}\sqrt{J_r+mR^2}}\partial_{p_1}
		-\frac{\sqrt{J_r}v_sp_1}{\sqrt{J_s}\sqrt{J_r+mR^2}}\partial_{p_2},
	\end{multline*}
	and the associated differential equations are
	\begin{equation}\label{eq:diskXsing}
		\begin{aligned}
			\dot{x}(t)=&\;\frac{R\cos(\theta(t))v_r(t)}{\sqrt{J_r+mR^2}},\\
			\dot{y}(t)=&\;\frac{R\sin(\theta(t))v_r(t)}{\sqrt{J_r+mR^2}},\\
			\dot{\theta}(t)=&\;\frac{v_s(t)}{\sqrt{J_s}},\\
			\dot{\phi}(t)=&\;\frac{v_r(t)}{\sqrt{J_r+mR^2}},\\
			\dot{v}_s(t)=&\;0,\\
			\dot{v}_r(t)=&\;\frac{mgR\sqrt{J_s}\sin(\tau)\cos(\theta(t))}
			{\sqrt{J_s}\sqrt{J_r+mR^2}},\\
			\dot{p}_1(t)=&\;\frac{\sqrt{J_r}v_s(t)p_2(t)}{\sqrt{J_s}\sqrt{J_r+mR^2}},\\
			\dot{p}_2(t)=&\;-\frac{\sqrt{J_r}v_s(t)p_1(t)}{\sqrt{J_s}\sqrt{J_r+mR^2}}.
		\end{aligned}
	\end{equation}
	
	One computes
	\begin{equation*}
		\scrL_{X_\D^{\mathrm{sing}}}\lambda^{\mathrm{e}}=
		\frac{2mgR\sin(\tau)\cos(\theta)}{\sqrt{J_r+mR^2}}v_rp_1
		+(v_s^2+v_r^2)\frac{\sqrt{J_r}}{\sqrt{J_s}\sqrt{J_r+mR^2}}v_sp_2.
	\end{equation*}
	We follow the same strategy as in the preceding paragraph to determine the
	cogeneralized subbundles that are invariant under $X_\D^{\mathrm{sing}}$\@.
	\begin{enumerate}[S1.,nosep]
		\item We first take $\sin(\tau)\not=0$\@.  If
		$p\in\A_0(X_\D^{\mathrm{sing}},\lambda^{\mathrm{e}})$\@, then either
		(a)~$p_1(p)=0$ or (b)~$v_s(p)=v_r(p)=0$\@.  We consider these possibilities
		in turn.
		\begin{enumerate}[(a),nosep]
			
			\item If $p_1(p)=0$\@, then
			$p\in\A_1(X_\D^{\mathrm{sing}},\lambda^{\mathrm{e}})$ if and only if
			\begin{equation*}
				v_s(p)p_2(p)(v_s(p)^2+v_r(p)^2)=0,
			\end{equation*}
			and this holds if and only if either (i)~$v_s(p)=0$ or (ii)~$p_2(p)=0$\@.  We
			consider these possibilities in turn.
			
			\begin{enumerate}[(i),nosep]
				\item In this case we see that
				\begin{equation*}
					\A=\{p\in\pi_\D^*\D^\perp\ |\ p_1(p)=0,\ v_s(p)=0\}
				\end{equation*}
				is an $X_\D^{\mathrm{sing}}$-invariant subbundle and the value of $p_2$ is
				constant along integral curves of $X_\D^{\mathrm{sing}}$ with initial
				conditions in $\A$\@.
				
				\item In this case, if we take
				\begin{equation*}
					\A=\{p\in\pi_\D^*\D^\perp\ |\ p_1(p)=0,\ p_2(p)=0\},
				\end{equation*}
				then $\A$ is an $X_\D^{\mathrm{sing}}$-invariant affine subbundle variety,
				and the values of $p_1$ and $p_2$ are zero along integral curves of
				$X_\D^{\mathrm{sing}}$ with initial conditions in $\A$\@.
			\end{enumerate}
			
			\item Let us denote
			\begin{equation*}
				\A'_1(X_\D^{\mathrm{sing}},\lambda^{\mathrm{e}})=
				\{p\in\pi_\D^*\D^\perp\ |\ v_s(p)=0,\ v_r(p)=0\}.
			\end{equation*}
			For $p\in\A'_1(X_\D^{\mathrm{sing}},\lambda^{\mathrm{e}})$ we compute
			\begin{equation*}
				\scrL^2_{X_\D^{\mathrm{sing}}}\lambda^{\mathrm{e}}(p)=
				\frac{2m^2g^2R^2\sin(\tau)^2}{J_r+mR^2}\cos(\theta(p))^2p_1(p).
			\end{equation*}
			Setting aside the condition that $p_1(p)=0$ which has already been
			considered, we see that for the expression on the right to vanish we must
			have $\cos(\theta(p))=0$\@.  One can then see that, if we define
			\begin{equation*}
				\A=\{p\in\pi_\D^*\D^\perp\ |\ v_s(p)=0,\ v_r(p)=0,\ \cos(\theta(p))=0\},
			\end{equation*}
			then $\A$ is an $X_\D^{\mathrm{sing}}$-invariant affine subbundle variety.
			We can additionally see that the values of $p_1$ and $p_2$ are constant along
			integral curves of $X_\D^{\mathrm{sing}}$ with initial conditions in $\A$\@.
		\end{enumerate}
		
		\item Now we consider the case of $\sin(\tau)=0$\@.  As above, for
		$p\in\A_0(X_\D^{\mathrm{sing}},\lambda^{\mathrm{e}})$\@, we have the two cases
		(a)~$p_1(p)=0$ or (b)~$v_s(p)=v_r(p)=0$\@, which we consider in turn.
		
		\begin{enumerate}[(a),nosep]
			\item If $p_1(p)=0$\@, then
			$p\in\A_0(X_\D^{\mathrm{sing}},\lambda^{\mathrm{e}})$ if and only if either
			(i)~$v_s(p)=0$ or (ii)~$p_2(p)=0$\@.  We consider the two preceding cases.
			
			\begin{enumerate}[(i),nosep]
				\item Here we define
				\begin{equation*}
					\A=\{p\in\pi_\D^*\D^\perp\ |\ p_1(p)=0,\ v_s(p)=0\},
				\end{equation*}
				and note that $\A$ is an $X_\D^{\mathrm{sing}}$-invariant affine subbundle
				variety.  The value of $p_2$ is constant along integral curves of
				$X_\D^{\mathrm{sing}}$ with initial conditions in $\A$\@.
				
				\item If we define
				\begin{equation*}
					\A=\{p\in\pi_\D^*\D^\perp\ |\ p_1(p)=0,\ p_2(p)=0\},
				\end{equation*}
				then $\A$ is an $X_\D^{\mathrm{sing}}$-invariant affine subbundle variety,
				and the values of $p_1$ and $p_2$ are zero along integral curves of
				$X_\D^{\mathrm{sing}}$ with initial conditions in $\A$\@.
			\end{enumerate}
			
			\item Finally, if we take
			\begin{equation*}
				\A=\{p\in\pi_\D^*\D^\perp\ |\ v_s(p)=0,\ v_r(p)=0\},
			\end{equation*}
			then $\A$ is an $X_\D^{\mathrm{sing}}$-invariant affine subbundle variety,
			and the values of $p_1$ and $p_2$ are constant along integral curves of
			$X_\D^{\mathrm{sing}}$ with initial conditions in $\A$\@.
		\end{enumerate}
	\end{enumerate}
\end{paragr}

\begin{paragr}
	Let us now assemble the detailed analysis of the preceding two paragraphs
	into final results.  We first consider the case where nonholonomic
	trajectories are also regular constrained variational trajectories.
	\begin{proposition}
		\begin{enumerate}[(i),nosep]
			\item When\/ $\sin(\tau)\not=0$\@, the following initial conditions give all
			nonholonomic trajectories that are regular constrained variational
			trajectories for suitable choices of\/ $p_1(0)$ and\/ $p_2(0)$\@:
			\begin{enumerate}[(a),nosep]
				\item $\sin(\theta(0))=0$\@,\/ $v_s(0)=0$\@, $v_r(0)=0$\@;
				\item $\cos(\theta(0))=0$\@,\/ $v_s(0)=0$\@,\/ $v_r(0)=0$\@;
				\item $\sin(\theta(0))=0$\@,\/ $v_s(0)=0$\@.
			\end{enumerate}
			\item When\/ $\sin(\tau)=0$\@, all initial conditions give nonholonomic
			trajectories that are regular constrained variational trajectories for
			suitable choices of\/ $p_1(0)$ and\/ $p_2(0)$\@.
		\end{enumerate}
	\end{proposition}
	
	Next we consider the case where nonholonomic trajectories are also singular
	constrained variational trajectories.
	\begin{proposition}
		\begin{enumerate}[(i),nosep]
			\item When\/ $\sin(\tau)\not=0$\@, the following initial conditions give all
			nonholonomic trajectories that are singular constrained variational
			trajectories for suitable choices of\/ $p_1(0)$ and\/ $p_2(0)$\@:
			\begin{enumerate}[(a),nosep]
				\item $v_s(0)=0$\@;
				\item $\cos(\theta(0))=0$\@,\/ $v_s(0)=0$\@,\/ $v_r(0)=0$\@.
			\end{enumerate}
			\item When\/ $\sin(\tau)=0$\@, the following initial conditions give all
			nonholonomic trajectories that are singular constrained variational
			trajectories for suitable choices of\/ $p_1(0)$ and\/ $p_2(0)$\@:
			\begin{enumerate}[(a),nosep]
				\item $v_s(0)=0$\@.
			\end{enumerate}
		\end{enumerate}
	\end{proposition}
	
	These results improve the analysis in the literature for this example in
	various ways.  For example, our analysis includes the singular case for the
	first time, even in the oft-studied case when
	$\sin(\tau)=0$~\autocite[Example~6.6]{JC/MdL/DMdD/SM:02}\@.  We also are able
	to carry out a more detailed analysis in the regular case when
	$\sin(\tau)\not=0$ than is carried out in~\autocite{NAL:22}\@.  For example,
	in \autocite{NAL:22} it is indicated that (here we paraphrase to convert to
	our terminology), ``except for the trivial case $v_s(t)=0$\@, for which the
	constraints are actually holonomic, regular constrained variational
	trajectories are never nonholonomic trajectories.''  As we have seen in our
	analysis, this ``trivial case'' is actually not quite trivial since it arises
	at singularities of the cogeneralized subbundle $\ker(\hat{F}_\D^*)$\@.
	These singularities preclude most nonholonomic trajectories with $v_s(t)=0$
	from being regular constrained variational trajectories.  It does, however,
	permit such nonholonomic trajectories to be \emph{singular} constrained
	variational trajectories.
\end{paragr}

\printbibliography
{\small\par\noindent
\textsc{Andrew D. Lewis\\
Professor at the Department of Mathematics and Statistics,\\
Queen's University at Kingston.\\
\emph{Email address:}} \texttt{andrew.lewis@queensu.ca}
\medskip\par\noindent
\textsc{Ahmed G. Shaltut\\
Graduate Student at the Department of Mathematics and Statistics,\\
Queen's University at Kingston.\\
\emph{Email address:}} \texttt{ahmed.shaltut@queensu.ca}}

\end{document}